\documentclass[a4 paper,11pt]{amsart}

\usepackage{amsmath} %
\usepackage{amssymb} %
\usepackage{amscd} %
\usepackage{amsthm} %
\usepackage{mathrsfs} %
\usepackage{euscript}
\usepackage{eufrak}
\usepackage{ulem}
\usepackage[alphabetic]{amsrefs}
\usepackage{ascmac} % 枠付き
\usepackage{comment}
\usepackage{bm}%太字

\usepackage[dvipdfmx]{graphicx} 

\usepackage{hyperref}
\usepackage{xcolor}
\definecolor{unbleu}{rgb}{0.03, 0.15, 0.4}

\hypersetup{
pdfborder = {0 0 0},
colorlinks,
linkcolor=unbleu,
citecolor=unbleu,
urlcolor=unbleu
}

\usepackage{fancyhdr}
 \newtheorem{theorem}{Theorem}[section]
 \newtheorem{lemma}[theorem]{Lemma}
 \newtheorem{proposition}[theorem]{Proposition}

\newtheorem{corollary}[theorem]{Corollary}

\newtheorem{maintheorem}{Theorem} %\theoremstyle{definition}の下におくと斜体にならない!
  
\theoremstyle{definition}
\newtheorem{definition}[theorem]{Definition}
\newtheorem{remark}[theorem]{Remark}

\newcommand{\Wdim}{{\rm Widim}}
\newcommand{\rdim}{{\rm rdim}}
\newcommand{\mdim}{{\rm mdim}}
\newcommand{\diam}{{\rm diam}}
\newcommand{\B}[1]{\boldsymbol{#1}}%

\begin{document}

\title[]{Rate distortion dimension of Gibbs measures for functions depending on the first two coordinates on the full shift of Ahlfors regular spaces}
%"Constraint Ergodic" -> "Constrained ergodic" 

\author[K. Inui]{Kanji Inui}
\address{Department of mechanical engineering and science, Faculty of science and engineering, Doshisha University}
\email{kinui@mail.doshisha.ac.jp}

\author[M. Shinoda]{Mao Shinoda}
\address{Department of Mathematics, Ochanomizu University, 2-1-1 Otsuka, Bunkyo-ku, Tokyo, 112-8610, Japan}
\email{shinoda.mao@ocha.ac.jp}

% \date{\today}
%\textcolor{black}{; Secondary 37J51}}
%Tsu20 など
%37A05 Dynamical aspects of measure-preserving transformations
%37B99 None of the above, but in this section
%94A34  	Rate-distortion theory in information and communication theory
%CPV
%37D35  	Thermodynamic formalism, variational principles, equilibrium states for dynamical systems
%28D20  	Entropy and other invariants
%37B40  	Topological entropy
%37C30  	Functional analytic techniques in dynamical systems; zeta functions, (Ruelle-Frobenius) transfer operators, etc.
%cf https://mathscinet.ams.org/msc/msc2020.html?t=28DXX&s=37C30&btn=Search&ls=Ct

%%%%%%%%%%%%%%%%%%%%%%%%%%%%%%%%%%%%%%%%%%%%%%%%%%%%%%%%%
%%%%%%%%%%%%%%%%%%%%%%%%%%%%%%%%%%%%%%%%%%%%%%%%%%%%%%%%%
\begin{abstract}
    The study on shift spaces in ergodic theory has extended beyond the classical setting, but there is room to discuss an extension of the Kolmogorov-Sinai entropy in the ergodic theoretical point of view.  
    On the other hand, the rate distortion dimension recently attracted attention in mean dimension theory because it behaves like the Kolmogorov-Sinai entropy on dynamical systems in the ``large" spaces in which the usual entropy is in general infinite. 
    According to these background, we investigate the connection between the Gibbs measure on the product spaces and the variational principle based on the rate distortion dimension. Concretely we calculate the rate distortion dimension of the Gibbs measure for a H\"older continuous potential depending on the first two coordinates and it satisfies the simplest case of themodynamic formalism based on the rate distortion dimension: the extension of the maximal measure of topological entropy. 
    Notably, to the best of our knowledge, this is the first study to introduce Gibbs measures into mean dimension theory, and this connection allows the mean dimension with potential to be computed more naturally through the notion of the zero-temperature limit, in the spirit of thermodynamic formalism.
\end{abstract}

\maketitle
% %\newpage

% \tableofcontents
%%%%%%%%%%%%%%%%%%%%%%%%%%%%%%%%%%%%%%%%%%%%%%%%%%%%%%%%

% \textcolor{red}{
% \begin{itemize}
%     \item $(\mathcal{X},d_{\mathcal{X}})$ with a Borel probability measure $\rho$.
%     w.l.g. we may assume $\diam \mathcal{X}=1$.
%     \item $(\mathcal{X}^{\mathbb{N}_0}, d)$
%     \item $\underline{x}\in \mathcal{X}^{\mathbb{N}_0}$
%     \item $\varphi(\underline{x})=\varphi(x_0,x_1)$
% \end{itemize}
% }

\section{Introduction}
    The study on shift spaces in ergodic theory has extended beyond the classical setting. 
    Baraviera et.al. consider general full shift of the countable product of a compact metric space, which is the extended version of classical settings studied by many researchers in ergodic theory (see for example \cite{LMST09, BCLMS2011, LMMS2015}).  
    They introduce Gibbs measures on the spaces and study a connection with statistical mechanics in mathematical physics. 
    Note that we need the a priori measure on the space to introduce the extended Ruelle operator (with respect to the a priori measure) in the above setting, which construct of the Gibbs measures, similar to the classical settings (for the classical setting, see \cite{Bow75}). 
    In addition, Cioletti et.al. extended the Gibbs measures with respect to the continuous potentials on the space. 
    They discuss the uniqueness criterion and thermodynamic formalism (see \cite{CLS20}). 

    However, there is room to discuss an extension of the Kolmogorov-Sinai entropy in the ergodic theoretical point of view, since the classical Kolmogorov-Sinai entropy on the countable product of a compact metric space does not work in general. 
    To address this issue, the papers \cite{LMST09}, \cite{BCLMS2011}, \cite{LW20} and \cite{CLS20} extend the themodynamic formalism from an optimization theoretic perspective, introducing a notion of entropy defined indirectly. While this approach is natural in view of convex analysis, the concrete computation of the entropy is not straightforward.
    
    On the other hand, the rate distortion dimension recently attracted attention in mean dimension theory because it behaves like the Kolmogorov-Sinai entropy on dynamical systems in the ``large" spaces in which the usual entropies is in general infinite. 
    The rate distortion dimension is a concept (with respect to a measure) originally introduced in information theory (\cite{KD1994}). 
    Tsukamoto and Lindenstrauss introduce the rate distortion dimension in the mean dimension theory (\cite{LinTsu18}) to point out a new connection between information theory and dynamical systems. 
    In addition, Lindenstrauss and Tsukamoto establish the double variational principle in \cite{LinTsu19} and Tsukamoto extends it to the double variational principle with the potential in \cite{Tsu20}, based on the technique in ergodic theory. 
    The above variational principles are interesting formulae in the mean dimension theory. 
    
   Against this background, a natural question arises: for dynamical systems on ``large" spaces, can the rate distortion dimension be regarded as a natural extension of the Kolmogorov-Sinai entropy from an ergodic-theoretic perspective?
    If so, then there are two advantages in the formulation based on the rate distortion dimension. 
    First, the notion of the rate distortion dimension is derived from the ergodic theory and the dimension theory. 
    This advantage may calculate its value concretely by applying the dimension theoretical technique.  
    Second, the definition of the rate distortion dimension is independent of the a priori measure. 
    This advantage may formulate the variational principles without the a priori measure and formulate the variational principles without the a priori measure.  
    Note that in the classical thermodynamic formalism, for a H\"older continuous function, there exists a unique invariant measure (called an equilibrium measure) which attains the topological pressure, and the unique measure is the Gibbs measure (see, \cite{D68} or \cite{Bow75}). 
    
    The aim of this paper is to investigate the connection between the Gibbs measure on the product spaces and the variational principle based on the rate distortion dimension. 
    To consider the connection, we specifically consider a slightly general setting than the examples considered in (\cite{LinTsu18}): the one-sided infinite product $\mathcal{X}^{\mathbb{N}_{0}}$ of compact Ahlfors regular spaces $\mathcal{X}$. 
    Typical examples of Ahlfors regular spaces are connected compact manifolds and self-similar sets and self-conformal sets satisfying the open set condition (see \cite{Fr14} and \cite{BSS2023}). 
    In particular, the case where the regularity exponent equals $0$ corresponds to $\mathcal{X}$ being a finite set, so that $\mathcal{X}^{\mathbb{N}_{0}}$ reduces to the classical one-sided full shift over a finite alphabet; in this sense, our setting is a natural extension of the classical theory.
    Next, as in \cite{LMST09}, \cite{BCLMS2011} and \cite{CLS20}, we introduce the Gibbs measure with respect to H\"older continuous potentials depending on the first two coordinates of $\mathcal{X}^{\mathbb{N}_{0}}$. 
    Note that \cite{LMST09} and \cite{BCLMS2011} present the case $\mathcal{X} = [0, 1]$: they consider the Gibbs measure for the H\"{o}lder continuous potential depending on the first two coordinates of $[0,1]^{\mathbb{N}_0}$. 
    Then, we concretely calculate the rate distortion dimension of the Gibbs measure on $\mathcal{X}^{\mathbb{N}_{0}}$ and it satisfies the simplest case of the themodynamic formalism based on the rate distortion dimension: the extension of the measure of maximal entropy.  
    
    The precise statement is the following: 
    let $\mathcal{X}^{\mathbb{N}_0}$ be the one-sided infinite product space where $\mathbb{N}_{0} := \{ \ n \in \mathbb{Z} \ | \ n \geq 0 \ \}$ and $\mathcal{X}$ an Ahlfors $s$-regular ($s \geq 0$) compact metric space (for the definition of Ahlfors regular space, see Subsection \ref{sebsec:Ahlfors_reg_and_rel_dim}). 
    Without loss of generality, the Ahlfors regularity implies the existence of a Borel probability measure $\rho$ satisfying the inequality (\ref{eq:Ahlfors_regular}) (for the inequality (\ref{eq:Ahlfors_regular}), see Subsection (\ref{sebsec:Ahlfors_reg_and_rel_dim})). 
    Henceforth, we call $\rho$ the a priori measure on $\mathcal{X}$. 
    If $\mathcal{X}=\mathbb{R}/\mathbb{Z}$, the shift space is also called the XY-model or the modified Bernoulli space. 
    Then, we set a metric $d$ on $\mathcal{X}^{\mathbb{N}_0}$ defined by
    \begin{align*}
        d(\underline{x},\underline{y})=\sum_{n\in \mathbb{N}_{0}}\frac{d_{\mathcal{X}}(x_n, y_n)}{2^{n}} 
    \end{align*}
    and the left shift $\sigma$ on $\mathcal{X}^{\mathbb{N}_0}$. 
    We assume the following assumption on potentials $\varphi \colon \mathcal{X}^{\mathbb{N}_0}\rightarrow \mathbb{R}$: $\varphi(\underline{x})=\varphi(\underline{y})$ if $x_i=y_i$ for $i=0,1$.
    Henceforth, we say that $\varphi$ depends on the first two coordinates if $\varphi$ satisfies the above assumption. 
    Abuse of notation, we write $\varphi(\underline{x})$ as $\varphi(x_0, x_1)$ \ ($\underline{x} = \{x_{i}\}_{i \in \mathbb{N}_{0}} \in \mathcal{X}^{\mathbb{N}_{0}} $). 
    Next, as in \cite{LMST09},  \cite{LMST09} or \cite{BCLMS2011}, we define the Gibbs measure $\mu_{\varphi}$ for $\varphi$ and the a priori measure $\rho$ (for details, see Subsection \ref{sebsec:Gibbs_meas}).
    Let ${\rm rdim}(\mathcal{X}^{\mathbb{N}_0},\sigma,d,\mu_{\varphi})$ be the rate distortion dimension with respect to $\mu_{\varphi}$, $\sigma$ and $d$ ( for details, see Subsection \ref{subsec:rate_distortion_dimension}). 
    Now, we state our first main theorem. 
    \begin{maintheorem}
    \label{main:Gibbs_RDD} 
        Let $\mathcal{X}^{\mathbb{N}_0}$ be the one-sided shift space of an Ahlfors $s$-regular \textcolor{red}{($s>0$) }compact space, with the metric $d$ and the left shift $\sigma$ defined as above. 
        %Let $\rho$ be a Borel probability measure on $\mathcal{X}$ satisfying \eqref{eq:Ahlfors_regular}.
        Then, for the Gibbs measure $\mu_\varphi$ with respect to the a priori measure $\rho$ and a H\"older continuous function $\varphi: \mathcal{X}^{\mathbb{N}_0}\rightarrow\mathbb{R}$ depending on the first two coordinates, we have
        \begin{align*}
            {\rm rdim}(\mathcal{X}^{\mathbb{N}_0},\sigma,d,\mu_{\varphi})=s.
        \end{align*}
    \end{maintheorem}

\begin{remark}
    To the best of our knowledge, all previously known explicit computations of the rate distortion dimension have been restricted to the i.i.d.\ setting.
    Theorem \ref{main:Gibbs_RDD} gives, to our knowledge, the first such computation for stationary Markov measures.
    The Ruelle operator formalism introduced in Subsection \ref{sebsec:Gibbs_meas} is the natural tool for treating Markov measures on the full shift over a general alphabet; it is perhaps for this reason that such measures had not previously been studied from this perspective.
\end{remark}

   Since the rate distortion dimension is the measure-theoretic counterpart of the variational principle for mean dimension, we now relate Theorem \ref{main:Gibbs_RDD} to the mean metric dimension (for the definition of the mean metric dimension $\mdim_M$, see Subsection \ref{subsec:mean_dimensions}).
    For an Ahlfors $s$-regular $\mathcal{X}$, we have $\dim_B \mathcal{X}=s$, and $\mathcal{X}$ satisfies the assumption of Theorem D in \cite{CPV24}, that is, every nonempty subset $U$ of $\mathcal{X}$ satisfies $\dim_B U=s$ .
    Hence, by Theorem D in \cite{CPV24} with $\varphi=0$, we have
    \begin{align}
        \mdim_M(\mathcal{X}^{\mathbb{N}_0}, \sigma, d)=s.
        \label{eq:mean_metric_TheoremD}
    \end{align}
Theorem \ref{main:Gibbs_RDD} together with the above leads to the simplest case of the variational principle based on the rate distortion dimension: the extension of the measure of maximal entropy.
    
    % \begin{remark}
    %     For an Ahlfors $s$-regular compact metric space $\mathcal{X}$, $\dim_B \mathcal{X}=s$ and $\mathcal{X}$ satisfies the assumption of Theorem D in \cite{CPV24}, that is, every nonempty subset $U$ of $\mathcal{X}$ satisfies $\dim_BU=s$.
    %     Hence $\mdim_M(\mathcal{X}^{\mathbb{N}_0}, \sigma, d)=s$ follows from Theorem D in \cite{CPV24} with $\varphi=0$. 
    % \end{remark}
    
    %
    %\textcolor{red}{:revised to here}
    \begin{corollary} \label{maincor:maximal_entropy_of_measures}
        Under the same setting and assumption on Theorem \ref{main:Gibbs_RDD}, we have
        \begin{align*}
            \mdim_M(\mathcal{X}^{\mathbb{N}_0}, \sigma, d) 
            = \sup_{\mu\in \mathcal{M}_\sigma(\mathcal{X}^{\mathbb{N}_0})} {\rm rdim}(\mathcal{X}^{\mathbb{N}_0},\sigma,d,\mu)
            = {\rm rdim}(\mathcal{X}^{\mathbb{N}_0},\sigma,d,\mu_{\varphi}).  
        \end{align*}
    \end{corollary}
    \begin{proof}
        Recall that the following inequality holds in general (see \cite{LinTsu18}): 
        \begin{align*}
            \mdim_M(\mathcal{X}^{\mathbb{N}_0}, \sigma, d) \geq \sup_{\mu\in \mathcal{M}_\sigma(\mathcal{X}^{\mathbb{N}_0})} {\rm rdim}(\mathcal{X}^{\mathbb{N}_0},\sigma,d,\mu), 
        \end{align*}
        where $\mathcal{M}_\sigma(\mathcal{X}^{\mathbb{N}_0})$ is the set of all $\sigma$-invariant Borel probability measures on $\mathcal{X}^{\mathbb{N}_{0}}$. 
        It follows from Theorem \ref{main:Gibbs_RDD} and \eqref{eq:mean_metric_TheoremD} that 
        \begin{align*}
            \mdim_M(\mathcal{X}^{\mathbb{N}_0}, \sigma, d) 
            = \sup_{\mu\in \mathcal{M}_\sigma(\mathcal{X}^{\mathbb{N}_0})} {\rm rdim}(\mathcal{X}^{\mathbb{N}_0},\sigma,d,\mu)
            = {\rm rdim}(\mathcal{X}^{\mathbb{N}_0},\sigma,d,\mu_{\varphi}).  
        \end{align*}
        Thus, we have proved our corollary. 
    \end{proof}
    \begin{remark}
        Corollary \ref{maincor:maximal_entropy_of_measures} shows a new phenomenon that does not occur in the classical setting: the uniqueness of the measure (the product of the uniform measure) which attains the supremum of the measure of maximal entropy for the full shift.  
        Indeed, as in the classical setting, if the potentials $\varphi$ and $\tilde{\varphi}$ on $\mathcal{X}^{\mathbb{N}_{0}}$ are not cohomologous, then the Gibbs measures $\mu_{\varphi}$ and $\mu_{\tilde{\varphi}}$ on $\mathcal{X}^{\mathbb{N}_{0}}$ are different (see Proposition \ref{prop:not_chomologous_different_Gibbs}). 
        On the other hand, Corollary \ref{maincor:maximal_entropy_of_measures} claims that for each H\"older potential $\varphi$ depending on the first two coordinates, the Gibbs measure $\mu_{\varphi}$ attains the supremum of the extension of the measure of maximal entropy: the non-uniqueness of the supremum of the extension of the measure of maximal entropy for the full shift. 
    \end{remark}

    %\textcolor{red}{Shinoda revised from here:}
    Combining Theorem \ref{main:Gibbs_RDD} with a zero-temperature argument for Gibbs measures, we also obtain the variational principle for the mean metric dimension with potential (for the definition of the mean metric dimension $\mdim_M$, see Subsection \ref{subsec:mean_dimensions}):
    \begin{corollary}
    \label{cor:mean_metric_RDD}
        Under the same setting and assumption on Theorem \ref{main:Gibbs_RDD}, we have
        \begin{align}
            \overline{\mdim_M}(\mathcal{X}^{\mathbb{N}_0}, \sigma, d, \varphi)
            &=\sup_{\mu\in \mathcal{M}_\sigma(\mathcal{X}^{\mathbb{N}_0})}\left(\underline{\rdim}(\mathcal{X}^{\mathbb{N}_0}, \sigma, d, \mu)+\int \varphi d\mu\right)\nonumber\\
            &=\sup_{\mu\in \mathcal{M}_\sigma(\mathcal{X}^{\mathbb{N}_0})}\left(\overline{\rdim}(\mathcal{X}^{\mathbb{N}_0}, \sigma, d, \mu)+\int \varphi d\mu\right)
            \label{eq:mean_metric_RDD_potential}
        \end{align}
        where $ \overline{\mdim_M}(\mathcal{X}^{\mathbb{N}_0}, \sigma, d, \varphi)$ is the mean metric dimension with potential.
    \end{corollary}
    \begin{proof}
        By Corollary 1.7 in \cite{Tsu20}, we have
        \footnote{To apply Corollary 1.7, we need to verify that the metric $d$ satisfies tame growth of covering numbers in Definition 4.1 in \cite{Tsu20}. See also Example 3.9 in \cite{LinTsu19}}
        \begin{align*}
            \overline{\mdim_M}(\mathcal{X}^{\mathbb{N}_0}, \sigma, d, \varphi)
            \geq \sup_{\mu\in \mathcal{M}_\sigma(\mathcal{X}^{\mathbb{N}_0})}\left(\overline{\rdim}(\mathcal{X}^{\mathbb{N}_0}, \sigma, d, \mu)+\int \varphi d\mu\right).
        \end{align*}
       It remains to show the reverse inequality.
        By Theorem D in \cite{CPV24} we have
        \begin{align}
            \overline{\mdim_M}(\mathcal{X}^{\mathbb{N}_0}, \sigma, d, \varphi)=\dim_B\mathcal{X}+\beta (\varphi)=s+\beta(\varphi) 
            \label{eq:mean_metric_explicit}
        \end{align}
        where $\displaystyle \beta(\varphi)=\sup_{\mu\in \mathcal{M}_\sigma(\mathcal{X}^{\mathbb{N}_0})} \int \varphi d\mu$.
         \footnote{In \cite{CPV24}, the supremum is taken over ergodic measures. 
Our definition is equivalent by a standard fact in ergodic optimization, 
where $\beta(\varphi)$ is called the maximum ergodic average.}
        By Theorem \ref{main:Gibbs_RDD}, for every $t>0$ we have
        \begin{align*}
            \rdim(\mathcal{X}^{\mathbb{N}_0}, \sigma, d, \mu_{t\varphi})+\int \varphi d\mu_{t\varphi}=s+\int \varphi d\mu_{t\varphi}
        \end{align*}
        where $\mu_{t\varphi}$ is the Gibbs measure of $t\varphi$.
        By Theorem 5 in \cite{LMMS2015} any accumulation point of $\{\mu_{t\varphi}\}$ is a maximizing measure of $\varphi$.
        Hence there exists a convergent subsequence $\{\mu_{t_i\varphi}\}$ of Gibbs measures such that 
        \begin{align*}
            \lim_{i\to\infty}\int \varphi d\mu_{t_i}=\beta(\varphi).
        \end{align*}
        For $\varepsilon>0$ and $i$ sufficiently large, we have
        \begin{align*}
            \overline{\mdim_M}(\mathcal{X}^{\mathbb{N}_0}, \sigma, d, \varphi)
            =s+\beta(\varphi)
            &<\rdim(\mathcal{X}^{\mathbb{N}_0}, \sigma, d, \mu_{t_i\varphi})+\int \varphi d\mu_{t_i\varphi}+\varepsilon\\
            &\leq\sup_{\mu\in \mathcal{M}_\sigma(\mathcal{X}^{\mathbb{N}_0})}\left(\underline{\rdim}(\mathcal{X}^{\mathbb{N}_0}, \sigma, d, \mu)+\int \varphi d\mu\right)+\varepsilon,
        \end{align*}
        which completes the proof.
    \end{proof}

    \begin{remark}
    In \cite{Tsu20}, the case $\mathcal{X}=[0,1]$ is computed as an example; equation \eqref{eq:mean_metric_RDD_potential} in Corollary \ref{cor:mean_metric_RDD} can be regarded as a generalization of this example to Ahlfors $s$-regular spaces $\mathcal{X}$.
    Moreover, such variational principles are typically established indirectly, via the coincidence of the mean metric dimension with potential and the mean Hausdorff dimension with potential.
    In contrast, our proof of Corollary \ref{cor:mean_metric_RDD} is direct, obtained through the Gibbs measures $\mu_{t\varphi}$ and their zero-temperature limit — a strategy that is natural from the point of view of thermodynamic formalism.
    \end{remark}

    Finally, we also give an implication for the mean dimension with potential.
    Let $\mdim(\mathcal{X}^{\mathbb{N}_0},\sigma,\varphi)$ be the mean dimension of $\mathcal{X}^{\mathbb{N}_{0}}$ with respect to $\sigma$ and potential $\varphi$ (for details, see Subsection \ref{subsec:mean_dimensions}). 
    Tsukamoto \cite{Tsu20} introduced the mean dimension with potential and established the following double variational principle for a continuous map $T:\mathcal{Y}\rightarrow \mathcal{Y}$ on a compact metric space $\mathcal{Y}$ with the marker property and a continuous function $\varphi:\mathcal{Y}\rightarrow \mathbb{R}$, 
    \begin{align*}
        \mdim(\mathcal{Y}, T, \varphi)&=\min_{d\in \mathscr{D}(\mathcal{Y})} \sup_{\mu\in \mathcal{M}_T(\mathcal{Y})}\left(\overline{\rdim}(\mathcal{Y}, T, d, \mu)+\int \varphi d\mu\right)\\
        &=\min_{d\in \mathscr{D}(\mathcal{Y})} \sup_{\mu\in \mathcal{M}_T(\mathcal{Y})}\left(\underline{\rdim}(\mathcal{Y}, T, d, \mu)+\int \varphi d\mu\right)
    \end{align*}
    where $\mathscr{D}(\mathcal{Y})$ denotes the set of metrics on $\mathcal{Y}$ compatible with the topology. 
    Now, we present the following second theorem, which  holds only for the special case $\mathcal{X} = [0,1]^D$  but requires only continuity of the potential, 
    without any additional regularity assumption.
%    Note that the essential part of our argument does not depend on the dimension of the underlying cube. 
    \begin{maintheorem}
    \label{main:mdim_potential}
        For a continuous function $\varphi:([0,1]^D)^{\mathbb{N}_0}\rightarrow \mathbb{R}$ we have 
        \begin{align}
            \mdim(([0,1]^D)^{\mathbb{N}_{0}}), \sigma, \varphi) 
            = D + \sup_{\mu\in \mathcal{M}_\sigma(([0,1]^D)^{\mathbb{N}_0})}\int \varphi d\mu.
            \label{eq:mdim_potential}
        \end{align}
        %where $\mathcal{M}_\sigma(([0,1]^D)^{\mathbb{N}_0})$ is the space of $\sigma$-invariant Borel probability measures.
    \end{maintheorem}

    \begin{remark}
        For a H\"older continuous function $\varphi$ depending only on the first two coordinates, as required by Corollary \ref{cor:mean_metric_RDD}, equations \eqref{eq:mean_metric_explicit}, \eqref{eq:mdim_potential} and Corollary \ref{cor:mean_metric_RDD} together yield
        \begin{align*}
             \mdim(([0,1]^D))^{\mathbb{N}_{0}}, \sigma, \varphi)
             &=\sup_{\mu\in \mathcal{M}_\sigma(\mathcal{X}^{\mathbb{N}_0})}\left(\underline{\rdim}(\mathcal{X}^{\mathbb{N}_0}, \sigma, d, \mu)+\int \varphi d\mu\right)\nonumber\\
            &=\sup_{\mu\in \mathcal{M}_\sigma(\mathcal{X}^{\mathbb{N}_0})}\left(\overline{\rdim}(\mathcal{X}^{\mathbb{N}_0}, \sigma, d, \mu)+\int \varphi d\mu\right).
        \end{align*}
        This provides further evidence that the double variational principle for mean dimension with potential holds beyond the case where the marker property is assumed, extending Example 1.3 in \cite{Tsu20}.
    \end{remark}

    \begin{remark}
    Taking $\varphi=0$ in Theorem \ref{main:mdim_potential} recovers $\mdim(([0,1]^D)^{\mathbb{N}_{0}}, \sigma)=D$, consistent with the computation of the mean dimension of the torus $\mathbb{T}^D=(\mathbb{R}/\mathbb{Z})^D$ in \cite[Section 6]{LinTsu19}; in this sense, Theorem \ref{main:mdim_potential} can also be viewed as an extension of that computation to the case with potential.
\end{remark}

    The rest of the paper is organized as follows: 
    In Section 2, we summarize the notions used in our paper: Ahlfors regurality, Mean dimensions, Gibbs measures, the Rate distortion dimension. 
    In Section 3, We give the proof of main theorems. 
    We first presents the proof of Theorem \ref{main:mdim_potential} in Subsection \ref{subsec:main_thm_mean_dimension}. 
    We next presents the proof of Theorem \ref{main:mean_Hausdorff_Minkowski} in Subsection \ref{subsec:main_thm_mean_hausdorff_metric}. 
    We finally presents the proof of Theorem \ref{main:Gibbs_RDD} in Subsection \ref{subsec:main_thm_rate_distort_dimension}. 
\section{Settings}

\subsection{Ahlfors regularity and related dimensions} \label{sebsec:Ahlfors_reg_and_rel_dim}

Let $s\geq 0$. 
A metric space $\mathcal{X}$ is {\it Ahlfors} $s${\it -regular} if there exist a Borel measure $\rho$ supported on $\mathcal{X}$ and a constant $c\geq 1$ such that for every $x\in \mathcal{X}$ and $0<r\leq\diam \mathcal{X}$ we have
    \begin{align}
        c^{-1}r^s \leq \rho(\overline{B}(x,r))\leq cr^s
        \label{eq:Ahlfors_regular}
    \end{align}
    where $\overline{B}(x,r)$ is the closed ball centered at $x$ with radius $r$ and $\diam \mathcal{X}$ is the diameter of $\mathcal{X}$.
    %\textcolor{red}{
    Note that \eqref{eq:Ahlfors_regular} implies that $\rho$ is finite and \eqref{eq:Ahlfors_regular} is preserved under multiplication of $\rho$ by a positive constant. 
    Hence, without loss of generality, we may assume that $\rho$ is a probability measure.
    %}
    Moreover, if $\mathcal{X}$ is Ahlfors $s$-regular, its Assouad, Hausdorff and packing dimensions all coincide and are equal to $s$ (see \cite{Fr14}).

Let $(\mathcal{X},d_\mathcal{X})$ be an Ahlfors $s$-regular compact metric space and let $\rho$ be a Borel probability measure satisfying \eqref{eq:Ahlfors_regular}.
It is easy to see that $\rho$ satisfies
\begin{align}
    \rho(\overline{B}(x,kr))\leq c k^s \rho(\overline{B}(x,r))
    \label{eq:homogeneous_upper_measure}
\end{align}
 for every $x\in \mathcal{X}, 0<r\leq \diam \mathcal{X}$ and $k\geq 1$ where $c>0$ is a new constant which does not depend on $x,r$ and $k$.
A Borel measure satisfying \eqref{eq:homogeneous_upper_measure} is called a $(c,s)${\it -homogeneous measure} in \cite{VK88}. % and \textcolor{red}{$(c,s)$-homogeneous} measure. is also {\it doubling}. 
In \cite{VK88} it is briefly illustrated that the existence of $(c,s)$-homogeneous measure gives the upper bound of the cardinality of separating sets. For $\varepsilon>0$ and a set  $F\subset \mathcal{X}$, a subset $E\subset F$ is a $\varepsilon${\it -separating set} if $d_\mathcal{X}(x,y)\geq \varepsilon$ whenever $x,y\in E$ and $x\neq y$.
Then \eqref{eq:homogeneous_upper_measure} implies that there exists a constant $C>0$ such that for every $x\in \mathcal{X}$, $r>0$ and $k\geq 1$, 
\begin{align}
    \#_{{\rm sep}}(\overline{B}(x,kr), d_\mathcal{X},\varepsilon)\leq Ck^s
    \label{eq:homogeneous_upper_space}
\end{align}
where $\#_{{\rm sep}}(F, d_\mathcal{X},\varepsilon)$ is the maximum cardinality of $\varepsilon$-separating set of $F$.
A bounded metric space satisfying \eqref{eq:homogeneous_upper_space} is called $(C,s)${\it -homogeneous} in \cite{VK88}.% and such space is also doubling.

For an Ahlfors $s$-regular space, the opposite analogies of \eqref{eq:homogeneous_upper_measure} and \eqref{eq:homogeneous_upper_space} hold:
the Borel probability measure $\rho$ satisfies
\begin{align}
    \rho(\overline{B}(x,kr))\geq c^{-2}k^s \rho(\overline{B}(x,r)).
    \label{eq:homogeneous_lower_measure}
\end{align}
It is established in \cite[Proposition 5]{BG2000}
\footnote{The result is proved in the more general setting of pseudo-metric spaces. }
that the existence of a measure satisfying \eqref{eq:homogeneous_lower_measure} implies that
there exists $C>0$ such that for $x\in \mathcal{X}$, $0<r\leq \diam \mathcal{X}$ and $k\geq 1$, 
\begin{align}
    \#_{{\rm sep}}(\overline{B}(x,kr), d_\mathcal{X},r)\geq C^{-1}k^s.
    \label{homogebeous_lower_space}
\end{align}

Now we show some lemmas needed to prove the main theorems.
\begin{lemma}
    \label{lem:upper_adj}
    Let $(\mathcal{X},d_{\mathcal{X}})$ be a $(C,s)$-homogeneous compact metric space.
    Take $\varepsilon>0$ and a maximal $\varepsilon$-separating set $\{x_1, \ldots, x_\ell\}$.
    Then 
    \begin{align}
        \max_{1\leq j\leq \ell}\#\{i\mid \overline{B}(x_i,\varepsilon)\cap \overline{B}(x_j,2\varepsilon)\}\leq C 3^s. \label{eq:upper_adj}
    \end{align}
\end{lemma}
\begin{proof}
    For each $j\in \{1,\ldots, \ell\}$ set $I_j=\{i\mid \overline{B}(x_i,\varepsilon)\cap \overline{B}(x_j,2\varepsilon)\}$.
    For $i\in I_j$, take $u_i\in \overline{B}(x_i,\varepsilon)\cap \overline{B}(x_j,2\varepsilon)$.
    Then we have $d_{\mathcal{X}}(x_i,x_j)\leq d_{\mathcal{X}}(x_i,u_i)+d_{\mathcal{X}}(u_i,x_j)\leq 3\varepsilon$.
    Hence we have $\{x_i\mid i\in I_j\}\subset \overline{B}(x_j,3\varepsilon)$.
    Since $\{x_i\mid i\in I_j\}$ is $\varepsilon$-separating and $(\mathcal{X},d_{\mathcal{X}})$ is $(C,s)$-homogeneous, we have
    \begin{align*}
        \#\{x_i\mid i\in I_j\}=\#I_j\leq C 3^s,
    \end{align*}
    which completes the proof.
\end{proof}

\begin{lemma}
    \label{lem:upper_sep}
    Let $(\mathcal{X},d_{\mathcal{X}})$ be $(C,s)$-homogeneous and take $\varepsilon_0\in(0,1)$ and $\lambda\geq 1$.
    For each $n\geq 0$ set $\varepsilon_n=\varepsilon_0\lambda^{-n}$ and let $\ell_n$ be the maximum of cardinality of $\varepsilon_n$-separating sets.
    Then we have
    \begin{align}
        \ell_n\leq (C \lambda^s)^n\ell_0 \quad \mbox{for all}\ n\geq 0.
        \label{eq:upper_sep}
    \end{align}
\end{lemma}
\begin{proof}
    Take a maximal $\varepsilon_0$-separating set $\{x^{(0)}_1,\ldots, x^{(0)}_{\ell_0}\}$.
    For each $1\leq i\leq \ell_0$, let $E_i\subset \overline{B}(x^{(0)}_i,\varepsilon_0)=\overline{B}(x^{(0)}_i,\lambda\varepsilon_{1})$ be a maximal $\varepsilon_1$-separating set.
    Then $\#E_i\leq C\lambda^s$ by homogeneity of $(\mathcal{X},d_{\mathcal{X}})$.
    Since 
    \begin{align*}
        \bigcup_{i=1}^{\ell_0}\bigcup_{x\in E_i}\overline{B}(x,\varepsilon_1)=\bigcup_{i=1}^{\ell_0}\overline{B}(x^{(0)}_i, 2\varepsilon_1)=\mathcal{X},
    \end{align*}
    we have $\ell_1\leq \ell_0 C\lambda ^s$.
   Applying the same argument inductively, we obtain $\ell_n\leq \ell_0 (C\lambda^s)^n$.
\end{proof}

For later use, we state the following lemma, which is essentially a reformulation of \eqref{eq:homogeneous_upper_measure} and \eqref{eq:homogeneous_lower_measure}.

\begin{lemma}
    \label{lem:measure_ball}
    Let $\rho$ be a probability measure on a compact metric space. 
    Take $\varepsilon_0\in (0,1)$ and $k\geq 1$.
    For each $n\geq 0$ set $\varepsilon_n=\varepsilon_0 k^{-n}$.
    If $\rho$ satisfies \eqref{eq:homogeneous_upper_measure}, we have
    \begin{align}
        \rho(\overline{B}(x,\varepsilon_n))\geq (c k^s)^{-n}\rho(\overline{B}(x,\varepsilon_0))
        \quad\mbox{for all}\ n\geq 0.
        \label{eq:upper_ball}
    \end{align}
    If $\rho$ satisfies \eqref{eq:homogeneous_lower_measure}, we have
    \begin{align}
        \rho(\overline{B}(x,\varepsilon_n))\leq  (c k^{-s})^{n}\rho(\overline{B}(x,\varepsilon_0))
        \quad\mbox{for all}\ n\geq 0.
        \label{eq:lower_ball}
    \end{align}
\end{lemma}

\subsection{Mean dimensions with potential} \label{subsec:mean_dimensions}
In this subsection, we define three notions of mean dimension: mean dimension with potential, mean metric dimension, and mean Hausdorff dimension.

Let $T:\mathcal{Y}\rightarrow\mathcal{Y}$ be a continuous map on a compact metric space $(\mathcal{Y}, d)$.
%Let $(\mathcal{Y},T)$ be a topological dynamical system with a metric $d$.
First, we define mean dimension with potential, which was introduced in \cite{Tsu20}.
Let $f: \mathcal{Y}\rightarrow \mathcal{Z}$ be a continuous map into some topological space $\mathcal{Z}$
and $\varphi:\mathcal{Y}\rightarrow\mathbb{R}$ be a continuous function.
For $\varepsilon>0$,  the map $f$ is an {\it  $\varepsilon$-embedding} if $\diam f^{-1}\{z\}<\varepsilon$ for every $z\in \mathcal{Z}$.
The {\it  $\varepsilon$-width dimension with potential} $\Wdim_\varepsilon(\mathcal{Y}, d, \varphi)$ is the infimum of 
\begin{align*}
   \max_{y\in \mathcal{Y}}(\dim_{f(y)}P+\varphi(y))
\end{align*}
where $P$ is a simplicial complex, $f: \mathcal{Y}\rightarrow P$ is an $\varepsilon$-embedding and $\dim_a P$ is the maximum of $\dim \Delta$ over all simplices $\Delta\subset P$ containing $a$.
For $N>0$,  set a metric $d_N$ and $S_N\varphi$ by
\begin{align*}
    d_N(x,y)=\max_{0\leq n\leq N-1} d(T^n x, T^n y)\ (x,y\in \mathcal{Y}),
    \quad
    S_N\varphi(x)=\sum_{n=0}^{N-1}\varphi(T^nx)\ (x\in \mathcal{Y}).
\end{align*}
Then we define the {\it  mean dimension with potential} by
\begin{align*}
\mdim(\mathcal{Y},T,\varphi)=\lim_{\varepsilon\to0}\left(\lim_{N\to\infty}\frac{\Wdim_\varepsilon(\mathcal{Y}, d_N, S_N\varphi)}{N}\right). 
\end{align*}
Note that the value of $\mdim(\mathcal{Y},T,\varphi)$ is independent of the choice of $d$.

%------
Next, we define the mean metric dimension with potential, following \cite{Tsu20}; for simplicity, we define the mean Hausdorff dimension without potential, as this suffices for our purposes.

For $\varepsilon>0$, set the $\varepsilon$-{\it covering number with potential} of $\mathcal{Y}$ with respect to the metric $d$ and the potential $\varphi$ by
\begin{align*}
    \#(\mathcal{Y},d,\varphi,\varepsilon)=\inf\left\{\sum_{i=1}^n \varepsilon^{-\sup_{U_i}\varphi}\ \middle|\ \bigcup_{i=1}^n U_i=\mathcal{Y}\ \mbox{and} \ \diam\ U_i<\varepsilon\ \mbox{for every}\ i\right\}.
\end{align*}
Then we set
\begin{align*}
    S(\mathcal{Y}, T, d, \varphi, \varepsilon)
    =\lim_{N\to\infty}\frac{1}{N}\log \#(\mathcal{Y}, d_N,\varphi, \varepsilon).
\end{align*}
The upper/lower mean metric dimension with potential is defined by
\begin{align*}
   \overline{\mdim}_M(\mathcal{Y}, T, d, \varphi)&=\varlimsup_{\varepsilon\to0}\frac{S(\mathcal{Y}, T, d, \varphi, \varepsilon)}{|\log \varepsilon|},\\
\underline{\mdim}_M(\mathcal{Y}, T, d, \varphi)&=\varliminf_{\varepsilon\to0}\frac{S(\mathcal{Y}, T, d, \varphi, \varepsilon)}{|\log \varepsilon|}.
\end{align*}
For $\varphi\equiv 0$, this recovers the mean metric dimension $\overline{\mdim}_M(\mathcal{Y},T,d)$ (resp. $\underline{\mdim}_M(\mathcal{Y},T,d)$) without potential.

%-----------
Next, we define the mean Hausdorff dimension.
For $\varepsilon>0$ and $\displaystyle t>\max_{y\in \mathcal{Y}}\varphi(y)$, set
\begin{align*}
    \mathcal{H}_\varepsilon^t(\mathcal{Y}, d, \varphi)
    =\left\{ \sum_{n=1}^\infty (\diam (U_n))^{t-\sup_{U_n}\varphi}\mid \bigcup_{n=1}^\infty U_n=\mathcal{Y}\ \mbox{and}\  \diam\ U_n<\varepsilon \ \mbox{for all}\ n\in \mathbb{N}\right\}.
\end{align*}
Define
\begin{align*}
    \dim_H(\mathcal{Y}, d, \varphi,\varepsilon)=\sup\{t\geq \max_{y\in \mathcal{Y}}\varphi(y)  \mid \mathcal{H}_\varepsilon^t(\mathcal{Y}, d, \varphi)\geq 1\}.
\end{align*}
%Note that $\displaystyle \dim_H(\mathcal{Y}, d)=\lim_{\varepsilon\to0} d_H(\mathcal{Y},d, \varepsilon )$.
% \begin{align*}
%     d_N(\underline{x},\underline{y})=\max_{0\leq i\leq N-1} d(\sigma^i \underline{x}, \sigma^i \underline{y}),
% \end{align*}
Then the mean Hausdorff dimension with potential is defined by
\begin{align*}
    \mdim_H(\mathcal{Y}, T, d,\varphi)
    =\lim_{\varepsilon\to0}\varlimsup_{N\to\infty}\frac{\dim_H(\mathcal{Y},d_N, S_N\varphi,\varepsilon)}{N}.
\end{align*}
% Remark that we can also consider the lower mean Hausdorff dimension by replacing $\displaystyle \varliminf_{n\to\infty}$ in the above definition.
For $\varphi\equiv 0$, this recovers the mean Hausdorff dimension $\mdim_H(\mathcal{Y},T,d)$.

Note that,  by Proposition 3.2 in \cite{LinTsu19} the mean Hausdorff dimension is always bounded above by the lower mean metric dimension:
\begin{align}
    \mdim_{H}(\mathcal{Y}, T, d)\leq \underline{\mdim}_M(\mathcal{Y},T, d)    
    \label{eq:mean_Hausdorff_mean_metric}
\end{align}

    Using the Ahlfors regularity of the base space $\mathcal{X}$ we have the following:
    \begin{proposition}
     \label{main:mean_Hausdorff_Minkowski}
        Let $\mathcal{X}^{\mathbb{N}_0}$ be the one-sided shift space of an Ahlfors $s$-regular \textcolor{red}{($s>0$) }compact metric space, with the metric $d$ and the left shift $\sigma$ defined as above. Then we have
        \begin{align*}
            \mdim_H(\mathcal{X}^{\mathbb{N}_0}, \sigma, d)=\mdim_M (\mathcal{X}^{\mathbb{N}_0}, \sigma, d)=s. 
        \end{align*}
    \end{proposition}

   The proof is fairly standard, so we defer it to the Appendix for the sake of completeness.

   \begin{remark}
   Proposition \ref{main:mean_Hausdorff_Minkowski} gives the value $\mdim_M(\mathcal{X}^{\mathbb{N}_0}, \sigma, d)=s$; together with Theorem \ref{main:Gibbs_RDD}, this is used in the proof of Corollary \ref{maincor:maximal_entropy_of_measures}.
    However, the content of Corollary \ref{maincor:maximal_entropy_of_measures} is not this value itself, but rather the fact that the supremum $\displaystyle \sup_{\mu\in\mathcal{M}_\sigma(\mathcal{X}^{\mathbb{N}_0})}{\rm rdim}(\mathcal{X}^{\mathbb{N}_0},\sigma,d,\mu)$ is attained by specific measures, namely the Gibbs measures $\mu_\varphi$.
\end{remark}

\subsection{Gibbs measures on the full shift of compact metric spaces} \label{sebsec:Gibbs_meas}

In this section we introduce the Gibbs measure for a potential depending on the first two coordinates.
We will present several results from \cite{BCLMS2011, LMST09} to define the Gibbs measures.
Gibbs measures for more general potentials are also studied in \cite{CLS20, LW20}.

Let $\varphi:\mathcal{X}^{\mathbb{N}_0}\rightarrow\mathbb{R}$ be a H\"older function which depends on the first two coordinates.
In this case, the Ruelle operator has a simple form.
\begin{definition}
    Define $\mathcal{L}, \overline{\mathcal{L}}: C(\mathcal{X})\rightarrow C(\mathcal{X})$ by
    \begin{align*}
        \mathcal{L}\psi(x_1)&=\int e^{ \varphi(x_0, x_1)}\psi(x_0) d\rho(x_0), \\
        \overline{\mathcal{L}}\psi(x_1)&=\int e^{\varphi(x_0,x_1)}\psi(x_1)\ d\rho(x_1)
    \end{align*}
    for $\psi \in C(\mathcal{X})$.
\end{definition}

Then we have the following:
\begin{theorem}[Slight modification of Theorem 3 in \cite{LMST09} and Theorem 15 in \cite{BCLMS2011}]
\label{thm:KR}
    The operators $\mathcal{L}, \overline{\mathcal{L}}$ have the same positive maximal eigenvalue $\lambda$, which is simple and isolated.
    Each eigenfunction of $\mathcal{L}$ and $\overline{\mathcal{L}}$ associated with $\lambda$ is positive.
\end{theorem}

%Let $\psi, \overline{\psi}$ be the positive eigenfunctions for $\mathcal{L}$ and $\overline{\mathcal{L}}$ associated to $\lambda$, which satisfy the normalization condition $\int \psi(\B{x})d\B{x}=\int \overline{\psi}(\B{x}) d\B{x}=1$. 
Let $\psi, \overline{\psi}$ be the positive eigenfunctions for $\mathcal{L}$ and $\overline{\mathcal{L}}$ associated with the maximal eigenvalue $\lambda$, with $\int \psi(x)d\rho(x)=\int \overline{\psi}(x) d\rho(x)=1$. 

Define a density function $\theta:\mathcal{X}\rightarrow \mathbb{R}$ by
\begin{align*}
    \theta(x) := \frac{\psi(x)\overline{\psi}(x)}{\pi} \quad (x\in \mathcal{X}),
\end{align*}
where $\pi=\int \psi(x)\overline{\psi}(x)d\rho(x)$, and a transition function $K:\mathcal{X}^2\rightarrow \mathbb{R}$ by
\begin{align*}
    K(x_0, x_1) := \frac{e^{\varphi(x_0, x_1)}\overline{\psi}(x_1)}{\overline{\psi}(x_0)\lambda} \quad (x_{0}, x_{1} \in \mathcal{X}).
\end{align*}

It is straightforward to show that $\theta$ is stationary for $K(\cdot, \cdot)$, i.e., 
\begin{align*}
    \theta(x_1)=\int \theta(x_0) K(x_0, x_1)  d\rho(x_0) d\rho(x_1),
\end{align*}
 for every $x_1\in \mathcal{X}$.
 For the initial probability measure $\theta$ and the transition $K$, one can define a stationary (canonical) Markov process $X=\{X_n\}_{n\in \mathbb{N}_{0}}$ with state space $\mathcal{X}$.
 This process determines a $\sigma$-invariant measure $\mu_\varphi$ on $\mathcal{X}^{\mathbb{N}_{0}}$ described by the process $X$.
Following \cite{BCLMS2011}, we call the measure $\mu_\varphi$ the Gibbs measure for $\varphi$.

As in the finite symbol case, potentials which are not cohomologous give different Gibbs measures.
H\"older functions (not necessarily depending on the first two coordinates) $\varphi_1$ and $\varphi_2$ are {\it cohomologous} if there exists a constant $K\in \mathbb{R}$ and a H\"older function $u$ such that $\varphi_1-\varphi_2=u-u\circ \sigma+K$ holds.

\begin{proposition}
\label{prop:not_chomologous_different_Gibbs}
    Let $\varphi_1, \varphi_2:\mathcal{X}^{\mathbb{N}_0}\rightarrow\mathbb{R}$ be H\"older potentials depending on the first two coordinates. 
    If $\varphi_1$ and $\varphi_2$ are not cohomologous, then $\mu_{\varphi_1}\neq\mu_{\varphi_2}$.
\end{proposition}

The following Lemma \ref{lem:cohomologous} and Lemma \ref{lem:periodic_const} immediately imply Proposition \ref{prop:not_chomologous_different_Gibbs}.

\begin{lemma}
\label{lem:cohomologous}
    Let $\varphi_1$ and $\varphi_2$ be H\"older functions (not necessarily depending on the first two coordinates).
    If there exists a constant $K\in \mathbb{R}$ such that for every periodic point $\underline{x}\in \mathcal{X}^{\mathbb{N}_0}$ we have
    \begin{align}
        S_m\varphi_1(\underline{x})-S_m\varphi_2(\underline{x})=mK
        \label{eq:periodic_constant}
    \end{align}
    where $m$ is the period of $\underline{x}$,
    then there exists a H\"older function $u:\mathcal{X}^{\mathbb{N}_0}\rightarrow \mathbb{R}$ such that
    \begin{align*}
        \varphi_1-\varphi_2=u-u\circ \sigma+K
    \end{align*}
     holds.
\end{lemma}
\begin{proof}
    By Lemma 1.29 in \cite{Bow75} there exists $\underline{x}\in \mathcal{X}^{\mathbb{N}_0}$ such that for every non-empty open set $U$ and $N\in \mathbb{N}$ there exists $n\geq N$ satisfying $\sigma^n(\underline{x})\in U$.
    Let $\Gamma=\{\sigma^n(\underline{x})\mid n\in \mathbb{N}\}$.
    Then we have $\overline{\Gamma}=\mathcal{X}^{\mathbb{N}_0}$.
    Set $\eta=\varphi_1-\varphi_2-K$
    and define $u:\Gamma\rightarrow \mathbb{R}$ by
    \begin{align*}
        u(\sigma^n(\underline{x}))=\sum_{i=0}^{n-1}\eta \circ \sigma^i (\underline{x})
        \quad (n\in \mathbb{N}).
    \end{align*}
    Note that $\eta$ is a H\"older function.
    Since $\overline{\Gamma}=\mathcal{X}^{\mathbb{N}_0}$ implies $\Gamma$ is infinite,
    $\sigma^n(\underline{x})\neq \sigma^{m}(\underline{x})$ if $n\neq m$ and $u$ is well-defined.

    We will show that $u$ is H\"older on $\Gamma$.
    Take $\underline{y}=\sigma^k(\underline{x}),\underline{z}=\sigma^{m}(\underline{x})\in \Gamma$.
    Without loss of generality, we may assume $k<m$.
    Set $\delta=d(\underline{y},\underline{z})>0$.
    By the choice of $\underline{x}$, there exists $\ell\in \mathbb{N}$ such that
    \begin{align*}
        \sigma^{\ell}(\underline{x})\in B(\underline{y},\delta)\cap B(\underline{z},\delta)
        \quad \mbox{and}\quad
        2^{-(\ell-m)}<\delta.
    \end{align*}
    Let $\underline{t}:=\sigma^\ell(\underline{x})$.
    Then 
    \begin{align*}
        |u(\underline{y})-u(\underline{z})|\leq |u(\underline{y})-u(\underline{t})|+|u(\underline{t})-u(\underline{z})|.
    \end{align*}
    Set 
    \begin{align*}
        \tilde{\underline{w}}:=x_0\cdots x_{k-1}(x_k\cdots x_{\ell-1})^\infty
        =x_0\cdots x_{k-1}(y_0\cdots y_{\ell-k-1})^\infty
    \end{align*}
    and $\underline{w}=\sigma^k(\tilde{\underline{w}})$.
    Since $\underline{w}$ is periodic with period $\ell-k$, the assumption implies 
    \begin{align*}
        \sum_{i=0}^{\ell-k-1}\eta\circ \sigma^i(\underline{w})
        =S_{\ell-k}\varphi_1(\underline{w})-S_{\ell-k}\varphi_2(\underline{w})-(\ell-k)K=0.
    \end{align*}
    Then we have
    \begin{align*}
        |u(\underline{y})-u(\underline{t})|
        &=\left|\sum_{i=k}^{\ell-1}\eta\circ \sigma^i(\underline{x})\right|\\
        &=\left|\sum_{i=0}^{\ell-k-1}\eta\circ \sigma^i(\underline{y})-\sum_{i=0}^{\ell-k-1}\eta\circ \sigma^i(\underline{w})\right|\\
        &\leq \sum_{i=0}^{\ell-k-1}\left|\eta(y_i\cdots y_{\ell-k-1}\underline{t})-\eta(y_i\cdots y_{\ell-k-1}\underline{w})\right|\\
        &\leq C_\eta\sum_{i=0}^{\ell-k-1}d(y_i\cdots y_{\ell-k-1}\underline{t},y_i\cdots y_{\ell-k-1}\underline{w})^\alpha\\
        &=C_\eta\sum_{i=0}^{\ell-k-1} \frac{1}{2^{\alpha(\ell-k-i)}}d(\underline{t},\underline{w})^\alpha\\
        &\leq C_\eta d(\underline{t},\underline{w})^\alpha
     \end{align*}
     where $C_\eta>0$ is the H\"older constant of $\eta$ and $\alpha$ is the H\"older exponent of $\eta$.
     By the choice of $\ell$ and $\underline{w}$, we have
     \begin{align*}
         d(\underline{t},\underline{w})
         \leq d(\underline{t},\underline{y})+d(\underline{y},\underline{w})
         \leq d(\underline{z},\underline{y})+2^{-(\ell-k)}
         <2d(\underline{z},\underline{y}).
     \end{align*}
     Then 
     \begin{align*}
         |u(\underline{y})-u(\underline{t})|\leq 2^\alpha C_\eta d(\underline{z},\underline{y})^\alpha.
     \end{align*}

     In the same way by setting $\underline{v}=x_0\cdots x_{m-1}(x_m\cdots x_{\ell-1})^\infty=x_0\cdots x_{m-1}(z_0\cdots z_{\ell-m-1})^\infty$, we have
     \begin{align*}
         |u(\underline{t})-u(\underline{z})|\leq C_\eta d(\underline{t},\underline{v})^\alpha
         \leq C_\eta(d(\underline{z},\underline{y})+2^{-(\ell-m)})^\alpha<2^\alpha C_\eta d(\underline{z},\underline{y})^\alpha.
     \end{align*}

     Hence $u$ is uniformly continuous on $\Gamma$ and therefore it is extended uniformly to a continuous function on $\overline{\Gamma}=\mathcal{X}^{\mathbb{N}_0}$.
     Abusing notation, we denote the extension also by $u$.
     Then we have
     \begin{align*}
         u-u\circ \sigma=\varphi_1-\varphi_2-K
     \end{align*}
     and $u$ is H\"older on $\mathcal{X}^{\mathbb{N}_0}$.
\end{proof}

\begin{lemma}
\label{lem:periodic_const}
    Let $\varphi_1$ and $\varphi_2$ be H\"older functions depending on the first two coordinates and let $\mu_{\varphi_1}$ and $\mu_{\varphi_2}$ be the Gibbs measures with respect to $\varphi_1$ and $\varphi_2$ respectively.
    If $\mu_{\varphi_1}=\mu_{\varphi_2}$, then there exists a constant $K\in \mathbb{R}$ such that \eqref{eq:periodic_constant} holds for every periodic point $\underline{x}\in \mathcal{X}^{\mathbb{N}_0}$.
    In particular, $\displaystyle K=\log \frac{\lambda_1}{\lambda_2}$ where $\lambda_i$ is the same maximal eigenvalue of $\mathcal{L}_{\varphi_i}$ and $\overline{\mathcal{L}}_{\varphi_i}$ for $i=1,2$.
\end{lemma}
\begin{proof}
    Let $\underline{x}\in \mathcal{X}^{\mathbb{N}_0}$ be a periodic point of period $m$.
    For $r>0$, set
    \begin{align*}
        B_r=\prod_{i=0}^{m}B(\sigma^i(\underline{x}),r)\times \mathcal{X}^{\mathbb{N}_0}.
    \end{align*}
    Note that $B(\underline{x},r)=B(\sigma^m(\underline{x}),r)$.
    For $i=1,2$, let $\lambda_i$ be the same maximal eigenvalue and $\psi_i, \overline{\psi}_i$ be the eigenfunctions obtained by Theorem \ref{thm:KR} for $\mathcal{L}_{\varphi_i}$ and $\overline{\mathcal{L}}_{\varphi_i}$.
    Also, let $\pi_i=\int \psi_i\overline{\psi_i} d\rho$ for $i=1,2$.

    Fix $\varepsilon>0$.
    Letting $r>0$ be sufficiently small,  we have
    \begin{align*}
        \left|S_m\varphi_i(x_0\cdots x_{m})-S_m\varphi_i(y_0\cdots y_m)\right|<\varepsilon,\\
        \left|\frac{\psi_i(x_0)}{\psi_i(y_0)}\right|\leq e^\varepsilon
        \quad\mbox{and}\quad
        \left|\frac{\overline{\psi_i}(x_m)}{\overline{\psi_i}(y_m)}\right|\leq e^\varepsilon
    \end{align*}
    for every $\underline{y}\in B_r$ and $i=1,2$.
    Then, for $i=1,2$,
    \begin{align*}
        \mu_{\varphi_i}(B_r)
        &=\int_{B_r} \psi_i(y_0)e^{S_m\varphi_i(y_0\cdots y_m)}\overline{\psi_i}(y_m) d\rho^{\otimes m+1}(y_0,\ldots, y_m)\\
%        &\leq \frac{e^{S_m\varphi_i(x_0\cdots x_m)+\varepsilon}}{\pi_i \lambda_i^m} \prod_{i=1}^{m-1}\rho(B(\sigma^i(\underline{x}),r)) \int_{B(\underline{x},r)\times B(\sigma^m(\underline{x}),r)} \psi_i(y_0)\psi_i(y_m) d\rho^{\otimes 2}(y_0,y_m)\\
        &\leq \frac{e^{S_m\varphi_i(x_0\cdots x_m)+\varepsilon}\psi_i(x_0)\overline{\psi_i}(x_0)e^{2\varepsilon}}{\pi_i \lambda_i^m} \prod_{i=0}^{m}\rho(B(\sigma^i(\underline{x}),r)).
    \end{align*}
    Note that $x_0=x_m$.
    Similarly, we have
    \begin{align*}
        \mu_{\varphi_i}(B_r)
        &\geq \frac{e^{S_m\varphi_i(x_0\cdots x_m)-\varepsilon}\psi_i(x_0)\overline{\psi_i}(x_0)e^{-2\varepsilon}}{\pi_i \lambda_i^m} \prod_{i=0}^{m}\rho(B(\sigma^i(\underline{x}),r)).
    \end{align*}

    Since $\mu_{\varphi_1}=\mu_{\varphi_2}$, by taking $\log$, we have
    \begin{align*}
        S_m\varphi_1(\underline{x})-S_m\varphi_2(\underline{x})
        &\geq \log \frac{\pi_1}{\pi_2}+m\log \frac{\lambda_1}{\lambda_2}-\log \frac{\psi_1(x_0)\overline{\psi_1}(x_0)}{\psi_2(x_0)\overline{\psi_2}(x_0)}-4\varepsilon.
    \end{align*}
    Since $\varepsilon>0$ is arbitrary and the opposite inequality is obtained by a similar argument, we have
    \begin{align}
        S_m\varphi_1(\underline{x})-S_m\varphi_2(\underline{x})
        =\log \frac{\pi_1}{\pi_2}+m\log \frac{\lambda_1}{\lambda_2}-\log \frac{\psi_1(x_0)\overline{\psi_1}(x_0)}{\psi_2(x_0)\overline{\psi_2}(x_0)}
        \label{eq:difference_dynamical_sum}
    \end{align}

    Let $k\in \mathbb{N}$.
    Since $\sigma^{mk}\underline{x}=\underline{x}$ and
    \begin{align*}
        S_{mk}\varphi_i(\underline{x})=kS_m\varphi_i(\underline{x})\ \quad (i=1,2)
    \end{align*}
    we have
    \begin{align*}
        k\left(S_m\varphi_1(\underline{x})-S_m\varphi_2(\underline{x})\right)=\log \frac{\pi_1}{\pi_2}+km\log \frac{\lambda_1}{\lambda_2}-\log \frac{\psi_1(x_0)\overline{\psi_1}(x_0)}{\psi_2(x_0)\overline{\psi_2}(x_0)}
    \end{align*}
    by \eqref{eq:difference_dynamical_sum}.
    Dividing by $k$ and letting $k\to\infty$, we have
    \begin{align*}
        S_m\varphi_1(\underline{x})-S_m\varphi_2(\underline{x})=m\log \frac{\lambda_1}{\lambda_2},
    \end{align*}
    which completes the proof.
\end{proof}

\subsection{Rate distortion dimension} \label{subsec:rate_distortion_dimension}
To conclude this section, we introduce the rate distortion dimension.
Let $(\Omega, \mathbb{P})$ be a probability space. 
Let $\mathcal{X}$ and $\mathcal{Y}$ be measurable spaces and $X:\Omega\rightarrow \mathcal{X}$ and $Y:\Omega\rightarrow \mathcal{Y}$ be measurable maps.
The {\it  mutual information} $I(X;Y)$ is defined by the supremum of 
\begin{align*}
\sum_{m=1}^M\sum_{n=1}^N \mathbb{P}((X,Y)\in P_m\times Q_n)\log \frac{\mathbb{P}((X,Y)\in P_m\times Q_n)}{\mathbb{P}(X\in P_m)\mathbb{P}(Y\in Q_n)}, 
\end{align*}
where $\{P_1, \ldots, P_M\}$ and $\{Q_1, \ldots, Q_N\}$ are partitions of $\mathcal{X}$ and $\mathcal{Y}$, respectively, with the convention that $0\log (0/a)=0$ for all $a\geq 0$.

Let $T:\mathcal{X}\rightarrow \mathcal{X}$ be a continuous map on a compact metric space $(\mathcal{X}, d_\mathcal{X})$.
%Let $(\mathcal{X},T)$ be a topological dynamical system with a metric $d_\mathcal{X}$.
Let $\mu$ be a $T$-invariant Borel probability measure on $\mathcal{X}$.
For $\varepsilon>0$, {\it the rate distortion function} $R(d,\mu,\varepsilon)$ is defined by the infimum of 
\begin{align*}
    \frac{I(X;Y)}{n},
\end{align*}
where $n\in \mathbb{N}$, $X$ and $Y=(Y_0, \ldots, Y_{n-1})$ are random variables defined on some probability space $(\Omega, \mathbb{P})$ such that $X$ takes values in $\mathcal{X}$, its law is given by $\mu$, and $Y_k$ takes values in $\mathcal{X}$ for every $0\leq k\leq n-1$, satisfying the {\it  distortion condition}
\begin{align}
    \mathbb{E}\left(\frac{1}{n}\sum_{k=0}^{n-1} d_\mathcal{X}(T^k X, Y_k)\right)<\varepsilon.
    \label{distortion_condition}
\end{align}
Here $\mathbb{E}(\cdot)$ is the expectation with respect to the probability measure $\mathbb{P}$.
The {\it  upper/lower rate distortion dimension} %$R_\mu(\varepsilon)$ 
is defined by
\begin{align*}
    \overline{\rdim}(\mathcal{X}, T, d, \mu)&=\limsup_{\varepsilon\to0}\frac{R(d,\mu, \varepsilon)}{|\log \varepsilon|},\\
    \underline{\rdim}(\mathcal{X}, T, d, \mu)&=\liminf_{\varepsilon\to0}\frac{R(d,\mu, \varepsilon)}{|\log \varepsilon|}.
\end{align*}
For simplicity, we write the upper and lower rate distortion dimensions 
as $\overline{\rdim}(\mu)$ and $\underline{\rdim}(\mu)$, respectively, 
when the dynamical system and the metric are clear from the context.

We use the following notation in the proof of the main theorems:
\begin{definition}
    \begin{align*}
        H(X)&=\sum_{x\in  \mathcal{X}}-\mathbb{P}(X=x)\log \mathbb{P}(X=x)\\
        H(X, Y)&=\sum_{x\in \mathcal{X}}\sum_{y\in \mathcal{Y}}-\mathbb{P}(X=x,Y=y)\log \mathbb{P}(X=x, Y=y)\\
        H(X\mid Y)&=\sum_{x\in \mathcal{X}}\sum_{y\in \mathcal{Y}}-\mathbb{P}(X=x,Y=y)\log \mathbb{P}(X=x \mid Y=y)
    \end{align*}
\end{definition}

We recall the following standard lemma on mutual information.
\begin{lemma}[Theorem 2.5.2 (Chain rule for mutual information) in \cite{CT06}]
\label{lem:chain_rule_mutual}
    Let $\{X_0,X_1, \ldots, X_N\}$ be a sequence of finite-valued random variables and $Y$ be a finite-valued random variable.
    Then we have
    \begin{align}
        I(X_0,\ldots, X_N; Y)&=H(X_0,\ldots, X_N)-H(X_0, \ldots, X_N\mid Y)\nonumber\\
        &=H(X_0)+\sum_{i=1}^{N-1}H(X_i\mid X_{0},\ldots, X_{i-1})\nonumber\\
            &\hspace{20pt}-\left(H(X_0\mid Y)+\sum_{i=1}^{N-1}H(X_i\mid X_0, \ldots, X_{i-1}, Y)\right)\nonumber\\
        &\geq H(X_0)+\sum_{i=1}^{N-1}H(X_i\mid X_{0},\ldots, X_{i-1})-\sum_{i=0}^{N-1}H(X_i\mid Y)
        \label{eq:chain_rule_mutual}
    \end{align}
\end{lemma}

\section{proof of main theorems}

\subsection{Mean dimension with a potential for $([0,1]^D)^{\mathbb{N}_0}$} \label{subsec:main_thm_mean_dimension}
We use the following fact on the maximum ergodic average:
\begin{proposition}[\cite{Jen06}]
    Let $T:\mathcal{Y}\rightarrow \mathcal{Y}$ be a continuous map on a compact metric space $\mathcal{Y}$.
    For a continuous function $\varphi:\mathcal{Y}\rightarrow \mathbb{R}$, we have
    \begin{align}
        \beta(\varphi)=\limsup_{n\to\infty} \sup_{y\in \mathcal{Y}}\frac{1}{n}S_n\varphi(y)
            =\sup_{y\in \mathcal{Y}}\limsup_{n\to\infty}\frac{1}{n}S_n\varphi(y), \label{max_ergodic_average}
    \end{align}
    where $\displaystyle \beta(\varphi) := \sup_{\mu\in \mathcal{M}_T(\mathcal{Y}) } \int_{\mathcal{Y}} \varphi \ \mathrm{d} \mu$. 
\end{proposition}

\begin{proof}[Proof of Theorem \ref{main:mdim_potential}]
    First, we show that $\mdim(([0,1]^D)^{\mathbb{N}_0}, \sigma,\varphi)$ is bounded above.
    Fix $\varepsilon>0$.
    Fix $k\geq 1$ with $2^{-k} < \varepsilon$ and $N\geq k+1$.
    By \eqref{max_ergodic_average}, without loss of generality, we may assume
    \begin{align*}
        \beta(\varphi)+\varepsilon>\sup_{\underline{\B{x}}\in([0,1]^D)^{\mathbb{N}_0}}\frac{1}{N}S_N\varphi(\underline{\B{x}}).
    \end{align*}
    Set $f:([0,1]^D)^{\mathbb{N}_0}\rightarrow([0,1]^D)^{k+N+1}$ by
    \begin{align*}
        %f(\underline{\B{x}})=\left(
        %\begin{pmatrix}x_{-k,1}\\\vdots\\x_{-k,D}\end{pmatrix}, \ldots, \begin{pmatrix}x_{0,1}\\\vdots\\x_{0,D}\end{pmatrix}, \ldots, \begin{pmatrix}x_{k+N,1}\\\vdots\\x_{k+N,D}\end{pmatrix}\right).
        f(\underline{\B{x}})
        = (\B{x}_{0}, \ldots, \B{x}_{k+N})
        =
        \left(\begin{pmatrix}x_{0,1}\\\vdots\\x_{0,D}\end{pmatrix}, \ldots, \begin{pmatrix}x_{k+N,1}\\\vdots\\x_{k+N,D}\end{pmatrix}\right).
    \end{align*}
    It is easy to see that $f$ is a $2^{-k}$-embedding.
    Hence we have
    \begin{align*}
        \Wdim_{2^{-k}}(([0,1]^D)^{\mathbb{N}_0}, d_N, S_N\varphi)
        &\leq \max_{\underline{\B{x}}\in ([0,1]^D)^{\mathbb{N}_0}}(\dim_{f(\underline{\B{x}})} ([0,1]^D)^{k+N+1}+S_N\varphi(\underline{\B{x}}))\\
        &\leq (k+N+1)D+N(\beta(\varphi)+\varepsilon).
    \end{align*}
    Dividing by $N$, letting $N\to\infty$ on both sides and then letting $\varepsilon\to0$, we have
    \begin{align*}
        \mdim(([0,1]^D)^{\mathbb{N}_{0}}, \sigma, \varphi)\leq D+\beta(\varphi).
    \end{align*}
    
    Second, we give the estimate from below.
    Fix $\varepsilon>0$.
    By \eqref{max_ergodic_average}, there exists $\underline{\B{z}}\in ([0,1]^D)^{\mathbb{N}_0}$ such that
    \begin{align*}
        \beta(\varphi)<\limsup_{n\to\infty}\frac{1}{n}S_n\varphi(\underline{\B{z}})+\varepsilon.
    \end{align*}
    By the uniform continuity of $\varphi$, take $\delta>0$ such that $d(\underline{\B{x}},\underline{\B{y}})\leq\delta$ implies $|\varphi(\underline{\B{x}})-\varphi(\underline{\B{y}})|<\varepsilon$
    for all $\B{x}, \B{y} \in ([0, 1]^{D})^{\mathbb{N}_{0}}$.
    For each $n\geq 1$, set 
    \begin{align*}
            Y_n=\Pi_{i=-\infty}^{-1}\{\B{z}_i\}\times\Pi_{i=0}^{n-1}I_i\times \Pi_{i=n}^\infty \{\B{z}_i\}, 
    \end{align*}
    %where $I_i$ is a closed cube containing $\B{z}_i$ whose sides are all the length $\delta$.
    where $I_i := \{ \ \B{x} \in [0, 1]^{D} \mid d_{\mathrm{max}}(\B{x}, \B{z}_{i}) \leq \delta \ \}$. 
    Note that for every $\underline{\B{y}}\in Y_n$ we have
    \begin{align*}
        d(\underline{\B{y}}, \underline{\B{z}})=\sum_{i=0}^{n-1}\frac{d_{\max}(\B{y}_i,\B{z}_i)}{2^{i+1}}
            \leq \delta.
    \end{align*}
    % Let $d_{\max}$ be a metric on $[0,1]^n$ defined by
    % \begin{align*}
    %     d_{\max}(x,y)=\max_{0\leq i\leq n-1}|x_i-y_i|
    % \end{align*}
    % for every $x=(x_i),y=(y_i)\in[0,1]^n$.
    By \cite[Lemma 2.2]{Tsu25}, we have
    \begin{align*}
        \Wdim_{\delta}(([0,1]^D)^{\mathbb{N}_0}, d_n, S_n\varphi)
        &\geq \Wdim_\delta(Y_n,d_n, S_n\varphi)\\
%        &\geq \Wdim_{\delta}(Y_n, d_n, S_n\varphi(\underline{\B{z}})-n\varepsilon)\\
        &\geq \Wdim_{\delta}(Y_n, d_n, S_n\varphi-n\varepsilon)\\
        &\geq \Wdim_{\delta}(([0,1]^D)^{n}, d_{\max})+S_n\varphi(\underline{\B{z}})-n\varepsilon\\
        &=nD+S_n\varphi(\underline{\B{z}})-n\varepsilon.
    \end{align*}
    Dividing by $n$ and letting $n\to\infty$, we have
    \begin{align*}
        \lim_{n\to\infty}\frac{\Wdim_\delta(([0,1]^D)^{\mathbb{N}_0}, d_n, S_n\varphi)}{n}\geq D+\limsup \frac{1}{n}S_n\varphi(\underline{\B{z}})-\varepsilon
        &\geq D+\beta(\varphi)-2\varepsilon.
    \end{align*}
    Letting $\delta\to0$ and $\varepsilon\to0$, we have $\mdim(([0,1]^D)^{\mathbb{N}_0}, \sigma, \varphi)\geq D+\beta(\varphi)$.
\end{proof}

\subsection{Proof of Theorem \ref{main:Gibbs_RDD}} \label{subsec:main_thm_rate_distort_dimension}

Since $\mu_\varphi$ is a $\sigma$-invariant Borel probability measure on $\mathcal{X}^{\mathbb{N}_0}$, the general inequality $$\mdim_M(\mathcal{X}^{\mathbb{N}_0},\sigma,d)\geq\sup_{\mu\in\mathcal{M}_\sigma(\mathcal{X}^{\mathbb{N}_0})}\overline{\rdim}(\mathcal{X}^{\mathbb{N}_0},\sigma,d,\mu)$$ (see \cite{LinTsu18}), together with Proposition \ref{main:mean_Hausdorff_Minkowski}, immediately gives
\begin{align*}
    \overline{\rdim}(\mu_\varphi)\leq s.
\end{align*}

Now we consider the lower bound of $\underline{\rdim}(\mu_\varphi)$.
Let $\varepsilon>0$.
For a maximal $\varepsilon$-separating set $\{x_1,\ldots, x_\ell\}$, set a partition function $p:\mathcal{X}\rightarrow \{1,\ldots, \ell\}$ by
\begin{align*}
    p(y)=i    \quad\mbox{if}\quad  y\in \widetilde{B}(x_i,\varepsilon)
\end{align*}
where $\widetilde{B}(x_1,\varepsilon)=\overline{B}(x_1,\varepsilon)$ and $\displaystyle\widetilde{B}(x_i,\varepsilon)=\overline{B}(x_i,\varepsilon)\setminus\left(\bigcup_{j=1}^{i-1}\overline{B}(x_j,\varepsilon)\right)$ for all $2\leq i\leq \ell$.

The following lemma holds without Ahlfors regularity.
\begin{lemma}
    \label{lem:upper_relative_entropy}
    Let $U$ and $V$ be $\mathcal{X}$-valued random variables.
    Take $\varepsilon>0$ and a maximal $\varepsilon$-separating set $\{x_1,\ldots, x_\ell\}$ and let $p:\mathcal{X}\rightarrow \{1,\ldots, \ell\}$ be the partition function with respect to $\{x_1,\ldots, x_\ell\}$.
    Then we have
    \begin{align}
        H(p(U)\mid p(V))\leq \log 2+ \log \max_{1\leq j\leq \ell} \#\{i\mid \overline{B}(x_i,\varepsilon)&\cap \overline{B}(x_j,2\varepsilon)\neq \emptyset\}\nonumber\\
        &+\mathbb{P}(d(U,V)>\varepsilon)\times \log \ell.
        \label{eq:upper_relative_entropy}
    \end{align}
\end{lemma}

\begin{proof}
    Set a $\{0,1\}$-valued random variable $E$ by
    \begin{align*}
        E= \left\{ \begin{array}{lll}
            1 &\mbox{if}& d(U,V)>\varepsilon,  \\
            0 & \mbox{else}.  
        \end{array} \right. 
    \end{align*}
    
    By the formula for the conditional entropy
    %\textcolor{red}{(cf.Lemma 2.4 in \cite{LinTsu18})},
    we have
    \begin{align}
        H(p(U)\mid p(V))
            &=H(E\mid p(V))+H(p(U)\mid E,p(V))-H(E\mid p(U),p(V))\nonumber\\
            &\leq H(E)+\mathbb{P}(E=0)H(p(U)\mid E=0,p(V))+\mathbb{P}(E=1)H(p(U)\mid E=1, p(V))\nonumber\\
            &\leq \log 2+H(p(U)\mid E=0,p(V))+\mathbb{P}(E=1)H(p(U))\nonumber\\
            &\leq \log 2+H(p(U)\mid E=0,p(V))+\mathbb{P}(d(U,V)>\varepsilon)\log \ell.\label{eq:conditional_entropy_upper}
    \end{align}

    \noindent
    \underline{Estimate on the middle term of \eqref{eq:conditional_entropy_upper}}\vspace{3pt}\\
    For each $1\leq j\leq \ell$, set $I_j=\{i\mid \overline{B}(x_i,\varepsilon)\cap \overline{B}(x_j,2\varepsilon)\neq \emptyset\}$.
    Since $E=0$ implies $d_{\mathcal{X}}(U,x_{p(V)})\leq d_{\mathcal{X}}(U,V)+d_{\mathcal{X}}(V,x_{p(V)})\leq 2\varepsilon$, we have $U\in \overline{B}(x_{p(U)},\varepsilon)\cap \overline{B}(x_{p(V)},2\varepsilon)\neq \emptyset$ and $p(U)\in I_{p(V)}$.
    % Set $I_E=\{(i,j)\mid i\in I_j, j\in I_i\}$.
    % Similarity, we have $p(V)\in I_{p(U)}$.
    Hence we have
    \begin{align*}
        H(p(U)\mid E=0,p(V))
        &=\sum_{i,j}-\mathbb{P}\begin{pmatrix}
            E=0\\p(U)=i\\p(V)=j
        \end{pmatrix}
        \log \frac{\mathbb{P}\begin{pmatrix}
            E=0\\p(U)=i\\p(V)=j
        \end{pmatrix}}{\mathbb{P}
        \begin{pmatrix}
            E=0\\p(V)=j
        \end{pmatrix}
        }\\
        &= \sum_{j=1}^\ell \mathbb{P}\begin{pmatrix}
            E=0\\p(V)=j
        \end{pmatrix}
        \sum_{i\in I_j}
        -\frac{\mathbb{P}\begin{pmatrix}
            E=0\\p(U)=i\\p(V)=j
        \end{pmatrix}}{
        \mathbb{P}\begin{pmatrix}
            E=0\\p(V)=j
        \end{pmatrix}
        }
        \log \frac{
        \mathbb{P}\begin{pmatrix}
            E=0\\p(U)=i\\p(V)=j
        \end{pmatrix}}{\mathbb{P}
        \begin{pmatrix}
            E=0\\p(V)=j
        \end{pmatrix}
        }\\
        &\leq \sum_{j=1}^\ell \mathbb{P}\begin{pmatrix}
            E=0\\p(V)=j
        \end{pmatrix} \log \#I_j\\
        &\leq \log \max_{1\leq j\leq \ell} \#I_j,
    \end{align*}
    which completes the proof.
\end{proof}

\begin{proof}[Proof of Theorem \ref{main:Gibbs_RDD}]
    Let $\mu$ be the Gibbs measure for $\varphi$ with a priori measure $\rho$ on $(\mathcal{X},d_{\mathcal{X}})$.

    Let $\varepsilon>0$.
    Let $X$ be an $\mathcal{X}^{\mathbb{N}_0}$-valued random variable whose law is $\mu$.
    Take $N\geq 1$ and a sequence of $\mathcal{X}^{\mathbb{N}_0}$-valued random variables $Y=(Y_0, Y_1, \ldots, Y_{N-1})$ with the distortion condition \eqref{distortion_condition}.

    Fix $L>1$ and $\kappa\geq 2$.
    Set $\varepsilon_n=L\kappa^{-n}$ for each $n\in \mathbb{N}$.
    For each $n$, take a maximal $\varepsilon_n$-separating set $\{x_1^{(n)},\ldots, x_{\ell_n}^{(n)}\}$ and a partition function $p=p_n:\mathcal{X}\rightarrow \{1,\ldots, \ell_n\}$.
    For each $i\in \{0,1,\ldots, N\}$, set $\widetilde{X}_i=p(X_i)$ and for $i\in \{0,1,\ldots, N-1\}$ set $\widetilde{Y}_{i,0}=p(Y_{i,0})$.

     Take $n\in \mathbb{N}$ such that $\kappa^{-n-1}\leq \varepsilon<\kappa^{-n}$.
     By the data processing inequality in \cite[Lemma 2.2]{LinTsu18}, 
%    and by Lemma \ref{superadditivity_Markov} 
    we have
    \begin{align*}
        \frac{1}{N}I(X;Y)&\geq \frac{1}{N} I(\widetilde{X}_0, \ldots, \widetilde{X}_{N-1}; Y)\\
            &\geq \frac{1}{N} I(\widetilde{X}_0, \ldots, \widetilde{X}_{N-1}; \widetilde{Y}_{0,0}, \ldots, \widetilde{Y}_{N-1,0}).
    \end{align*}
    By Lemma \ref{lem:chain_rule_mutual} and the Markov property of $X$, we have
    \begin{align*}
         \frac{1}{N}&I(X;Y)\\
         &\geq \frac{1}{N}\left(H(\widetilde{X}_0)+(N-1)H(\widetilde{X}_1\mid \widetilde{X}_0) \right)
         -\frac{1}{N}\sum_{m=0}^{N-1}H(\widetilde{X}_m\mid \widetilde{Y}_{m,0}).
    \end{align*}
   Since for every $m\in\{0,1,\ldots, N-1\}$, we have
    \begin{align*}
        \mathbb{P}\left(d_{\max}(X_m,Y_{m,0})> \varepsilon_n\right)
        \leq \frac{\mathbb{E}\left(d_{\max}(X_m,Y_{m,0})\right)}{\varepsilon_n},
    \end{align*}
    by the distortion condition \eqref{distortion_condition} and \eqref{eq:upper_relative_entropy} in Lemma \ref{lem:upper_relative_entropy}, we have
    \begin{align*}
        \frac{1}{N}\sum_{i=0}^{N-1}&H(\widetilde{X}_m|\widetilde{Y}_{m,0})\\
        &\leq \log 2+\log \max_{1\leq j\leq \ell_n}\#I_j+
        \frac{\log \ell_n}{N}\sum_{m=0}^{N-1} \mathbb{P}\left(d_{\max}(X_m,Y_{m,0})>\varepsilon_n \right)\\
        &\leq \log 2+\log \max_{1\leq j\leq \ell_n}\#I_j+\frac{\log \ell_n}{N}\sum_{m=0}^{N-1}\frac{\mathbb{E}\left(d_{\max}(X_m,Y_{m,0})\right)}{\varepsilon_n}\\
        &=\log 2+\log \max_{1\leq j\leq \ell_n}\#I_j+\frac{\log \ell_n}{\varepsilon_n}\mathbb{E}\left(\frac{1}{N}\sum_{m=0}^{N-1}\left(d_{\max}(X_m,Y_{m,0})\right)\right)\\
        &\leq \log 2+\log \max_{1\leq j\leq \ell_n}\#I_j+\frac{\log \ell_n}{\varepsilon_n}\mathbb{E}\left(\frac{1}{N}\sum_{m=0}^{N-1}d(\sigma^mX, Y_m )\right)\\
        &\leq \log 2+\log \max_{1\leq j\leq \ell_n}\#I_j+\frac{\log \ell_n}{L\kappa^{-n}}\varepsilon 
        \quad (\because \eqref{distortion_condition})\\
        &\leq \log 2+\log \max_{1\leq j\leq \ell_n}\#I_j+\frac{\log \ell_n}{L}.
        \quad (\because \varepsilon<\kappa^{-n})
    \end{align*}
    Moreover, by \eqref{eq:upper_adj} in Lemma \ref{lem:upper_adj}, we have
    \begin{align}
        \frac{1}{N}\sum_{i=0}^{N-1}H(\widetilde{X}_m|\widetilde{Y}_{m,0})
            \leq \log 2+ \log C3^s+\frac{\log \ell_n}{L}
        \label{eq:sum_relative_entorpy}
    \end{align}

     Since $\psi, \overline{\psi}$ are positive and uniformly continuous on $Y$, set $\displaystyle \psi_{\rm min}=\min\{\min_{x\in \mathcal{X}}\psi(x), \min_{x\in \mathcal{X}}\overline{\psi}(x)\}>0$.
    Note that $\overline{B}(x_i^{(n)},\kappa^{-1}\varepsilon_n)\cap \overline{B}(x_j^{(n)},\kappa^{-1}\varepsilon_n) =\emptyset$ whenever $i\neq j$ since $\{x_1^{(n)},\ldots, x_{\ell_n}^{(n)}\}$ is $\varepsilon_n$-separating.
    Hence for each $i,j\in \{1,\ldots, \ell_n\}$, we have
    \begin{align}
        \psi_{\rm min} \rho(\overline{B}(x_i^{(n)},\kappa^{-1}\varepsilon_n))
        &\leq \int_{\overline{B}(x_i^{(n)}, \varepsilon_n)}\theta (x) dx  
        \label{eq: lower_ball_1}\\
        &=\frac{1}{\pi} \int_{\overline{B}(x_i^{(n)}, \varepsilon_n)}\psi(x)\overline{\psi}(x) dx \nonumber\\
        &\leq \frac{1}{\pi} \|\psi\|_\infty \|\overline{\psi}\|_\infty \rho(\overline{B}(x_i^{(n)}, \varepsilon_n)) \nonumber
    \end{align}
    and
    \begin{align}
        \int \int _{\overline{B}(x_i^{(n)}, \varepsilon_n)\times \overline{B}(x_j^{(n)}, \varepsilon_n)}& \theta (x_1)K(x_1,x_2) dx_1 dx_2 \nonumber\\
        &=\frac{1}{\pi \lambda}\int \int _{\overline{B}(x_i^{(n)}, \varepsilon_n)\times \overline{B}(x_j^{(n)}, \varepsilon_n)}\psi(x_1) e^{\varphi(x_1, x_2)}\overline{\psi}(x_2) dx_1dx_2 \nonumber\\
 %       &\leq \frac{1}{\pi\lambda}  \|\psi\|_\infty \|\overline{\psi}\|_\infty \rho(\overline{B}(x_i^{(n)}, \varepsilon_n))\rho(\overline{B}(x_j^{(n)}, \varepsilon_n))\\
        &\leq \frac{1}{\pi\lambda}  \|\psi\|_\infty \|\overline{\psi}\|_\infty \rho(\overline{B}(x_i^{(n)}, \varepsilon_n))\rho(\overline{B}(x_j^{(n)}, \varepsilon_n))
        \label{eq:upper_ball_2}
    \end{align}

    Since $x\mapsto -\log x$ is decreasing, we have
    \begin{align}
       H(\widetilde{X}_1\mid \widetilde{X}_0) &=H(\widetilde{X}_0, \widetilde{X}_1)-H(\widetilde{X}_0) \nonumber\\
            &=\sum_{i,j}-\mathbb{P}(\widetilde{X}_0=i, \widetilde{X}_1=j)\log \frac{\mathbb{P}(\widetilde{X}_0=i, \widetilde{X}_1=j)}{\mathbb{P}(\widetilde{X}_0=i)}\nonumber\\
            &\geq \sum_{i,j}-\mathbb{P}(\widetilde{X}_0=i, \widetilde{X}_1=j)
            \log \frac{\frac{1}{\pi \lambda}\|\psi\|_\infty \|\overline{\psi}\|_\infty e^{\|\varphi\|_\infty}\rho(\overline{B}(x_i^{(n)}, \varepsilon_n))\rho(\overline{B}(x_i^{(n)}, \varepsilon_n))}{\psi_{\rm min}\rho(\overline{B}(x_i^{(n)}, \kappa^{-1}\varepsilon_n))}\nonumber\\
            &\geq C+\sum_{i,j}-\mathbb{P}(\widetilde{X}_0=i, \widetilde{X}_1=j)\log \frac{c\kappa^s\rho(\overline{B}(x_i^{(n)}, \kappa^{-1}\varepsilon_n))\rho(\overline{B}(x_j^{(n)}, \varepsilon_n))}{\rho(\overline{B}(x_i^{(n)}, \kappa^{-1}\varepsilon_n))}
            \quad (\because \eqref{eq:homogeneous_lower_measure}) \nonumber\\
            &\geq C-\log c\kappa^s+\sum_{j}-\mathbb{P}(\widetilde{X}_1=j) \log \rho(\overline{B}(x_j^{(n)},\varepsilon_n) )\nonumber\\
            &\geq C-\log c\kappa^s+\sum_{j}-\mathbb{P}(\widetilde{X}_1=j) \log (c\kappa^{-s})^{n} \rho(\overline{B}(x_j^{(n)},L))
            \quad (\because \eqref{eq:lower_ball}\ \mbox{in Lemma}\ \ref{lem:measure_ball}) \nonumber\\
            &\geq C+(n-1) \log (c\kappa^s)+\log \rho_L
            \label{eq:lower_relative_X1X0}
    \end{align}
    where $C=-\log (\pi \lambda \psi_{\rm min })^{-1}\|\psi\|_\infty \|\overline{\psi}\|_\infty e^{\|\varphi\|_\infty}$ and $\rho_L=\min_{x\in Y}\rho(\overline{B}(x,L))$.

    Combining \eqref{eq:sum_relative_entorpy} and \eqref{eq:lower_relative_X1X0}, we have
    \begin{align*}
        \frac{I(X;Y)}{N}
        &\geq \frac{1}{N}H(\widetilde{X}_0)+\frac{N-1}{N}H(\widetilde{X}_1\mid \widetilde{X}_0)-\frac{1}{N}\sum_{m=0}^{N-1}H(\widetilde{X}_m\mid \widetilde{Y}_{m,0})\\
        &\geq \frac{1}{N}H(\widetilde{X}_0)
        -\frac{N-1}{N}\left(C+(n-1) \log (c\kappa^s)+\log \rho_L\right)
        -\left(\log 2+ \log C3^s+\frac{\log \ell_n}{L}\right).
    \end{align*}

    Hence we have
    \begin{align*}
        \frac{R_\mu(\varepsilon)}{|\log \varepsilon|}
        &\geq \frac{1}{(n+1)\log \kappa}\left(\left(C+(n-1) \log (c\kappa^s)+\log \rho_L\right)
        -\left(\log 2+ \log C3^s+\frac{\log \ell_n}{L}\right)\right)\\
        &\geq\frac{\log c +s\log \kappa}{\log \kappa}\frac{n-1}{n+1}+\frac{C+\log \rho_L-\log 2+\log C3^s}{(n+1)\log \kappa}-\frac{\log (C\kappa^s)^n \ell_0}{(n+1)L\log \kappa}
        \quad (\because \eqref{eq:upper_sep})\\
        &=\frac{\log c +s\log \kappa}{\log \kappa}\frac{n-1}{n+1}+\frac{C+\log \rho_L-\log 2+\log C3^s-L^{-1}\log \ell_0}{(n+1)\log \kappa}-\frac{n}{n+1}\frac{\log (C\kappa^s)}{L\log \kappa}
    \end{align*}
    
    Taking $\liminf_{\varepsilon \to 0}\ (\liminf_{n\to\infty})$ on both sides we have
    \begin{align*}
        \underline{\rdim}(\mu_\varphi)\geq \frac{\log c +s\log \kappa}{\log \kappa}+0-\frac{\log (C\kappa^s)}{L\log \kappa}
    \end{align*}
    Letting $\kappa\to\infty$ and $L\to\infty$, we complete the proof.
\end{proof}

%----------------------------------------------
% **********************************************************
% Appendices
% **********************************************************

\setcounter{equation}{0}
 \renewcommand{\theequation}{\Alph{section}.\arabic{equation}}

\appendix

%------------------------
\section{Proof of Proposition \ref{main:mean_Hausdorff_Minkowski}} \label{subsec:main_thm_mean_hausdorff_metric}
% Throughout this section, let $(\mathcal{Y},d_{\mathcal{Y}})$ be a compact metric space and $T:\mathcal{Y}\rightarrow \mathcal{Y}$ be a continuous map.
% For $N \in \mathbb{N}$ set
% \begin{align*}
%     d_N(x,y)=\max_{0\leq n\leq N-1} d_{\mathcal{Y}}(T^nx, T^ny).
% \end{align*}

%\subsubsection{Lower bound of Mean Hausdorff dimension} 

For the full shift over an Ahlfors $s$-regular space, we have the following:
\begin{proposition}
\label{prop:mean_Hausdorff}
    \begin{align*}
        \mdim_H(\mathcal{X}^{\mathbb{N}_0}, \sigma, d)\geq s.
    \end{align*}
\end{proposition}

Proposition \ref{main:mean_Hausdorff_Minkowski} follows from 
Proposition \ref{prop:mean_Hausdorff}, the inequality \eqref{eq:mean_Hausdorff_mean_metric} and 
$\mdim_M(\mathcal{X}^{\mathbb{N}_0}, \sigma, d)=\dim_B\mathcal{X}=s$, 
which is given by Theorem D in \cite{CPV24}.

Since $\mathcal{X}$ is Ahlfors $s$-regular, we have a probability measure $\rho$ satisfying \eqref{eq:Ahlfors_regular}.
Considering the product measure of $\rho$, we obtain the lower bound by an argument similar to Frostman's lemma.
%We omit the proof here but include it in the appendix for the sake of completeness, as it follows from a standard argument.

%We now give the proof of Proposition \ref{prop:mean_Hausdorff}.

\begin{proof}[Proof of Proposition \ref{prop:mean_Hausdorff}]
    Let $(\mathcal{X}, d_{\mathcal{X}})$ be an Ahlfors $s$-regular compact metric space.
    Let $\rho$ be the probability measure satisfying \eqref{eq:Ahlfors_regular} and $\mu=\rho^{\otimes \mathbb{N}_0}$ be the product measure on $\mathcal{X}^{\mathbb{N}_0}$.
    
    Fix $N\geq 1$ and $\varepsilon$.
    Take an open cover $\{U_i\}_{i=1}^m$ with $\diam U_i=:R_i<\varepsilon$.
    For each $i$ take $\underline{x}^{(i)}\in U_i$.
    Then we have
    \begin{align*}
        U_i\subset B(\underline{x}^{(i)}, R_i; d_N)
        \subset \prod_{n=0}^{N-1} B(x_n^{(i)}, R_i;d_\mathcal{X})\times \mathcal{X}^{\mathbb{N}_0}.
    \end{align*}
    Hence we have
    \begin{align*}
        \mu(U_i)\leq \prod_{n=0}^{N-1} \rho(B(x_n^{(i)}, R_i;d_{\mathcal{X}}))\leq C^N R_i^{sN}.
    \end{align*}

    Then we have
    \begin{align*}
        1\leq \mu(\mathcal{X}^{\mathbb{N}_0})\leq \sum_{i=1}^{m}\mu(U_i)\leq C^N \sum_{i=1}^{m} R_i^{sN}.
    \end{align*}
    Take $\delta>0$.
    Since $R_i^s=R_i^{s-\delta}R_i^\delta\leq R_i^{s-\delta}\varepsilon^\delta$ for all $i$, we have
    \begin{align}
        1\leq C^N \varepsilon^{\delta N}\sum_{i=1}^m R_i^{(s-\delta)N}.
    \end{align}
    Since the open set is arbitrary, we have
    \begin{align*}
        1\leq C^N\varepsilon^{\delta N} \mathcal{H}_\varepsilon^{(s-\delta)N}(\mathcal{X}^{\mathbb{N}_0}, d_N).
    \end{align*}
   Taking $\varepsilon>0$ small enough to satisfy $C^N\varepsilon^{\delta N}<1$, we have $ \mathcal{H}_\varepsilon^{(s-\delta)N}(\mathcal{X}^{\mathbb{N}_0}, d_N)>1$.
   Hence we have
   \begin{align}
       \dim_H(\mathcal{X}^{\mathbb{N}_0}, d_N, \varepsilon)\geq (s-\delta)N.
       \label{eq: Hausdorff_fullshift_dN}
   \end{align}

   Dividing \eqref{eq: Hausdorff_fullshift_dN} by $N$, letting $\varepsilon\to0$, we have
   \begin{align*}
       \mdim_H(\mathcal{X}^{\mathbb{N}_0}, \sigma, d)\geq s-\delta.
   \end{align*}
   Since $\delta>0$ is arbitrary, we complete the proof.
\end{proof}

%\subsubsection{Upper bound of Mean metric dimension}

Although the following proposition is not needed for the proof of 
Proposition \ref{main:mean_Hausdorff_Minkowski} thanks to Theorem D 
in \cite{CPV24}, we include it for completeness, as it follows 
straightforwardly from the Ahlfors regularity.

\begin{proposition}
\label{prop:mean_metric}
    \begin{align*}
        \mdim_M(\mathcal{X}^{\mathbb{N}_0}, \sigma, d)\leq s
    \end{align*}
\end{proposition}

\begin{proof}[Proof of Proposition \ref{prop:mean_metric}]
Fix $\varepsilon>0$.
Take $m\in \mathbb{N}$ such that $\displaystyle \sum_{i\geq m}2^{-(i+1)}\leq \varepsilon/2 $.
Let $\{x_1,\ldots, x_\ell\}$  be a minimal $\varepsilon/2$-spanning set, that is, $\displaystyle \mathcal{X}\subset \bigcup_{i=1}^\ell B(x_i, \varepsilon/2)$ and $\ell$ is minimal with this property.

Fix $N\geq 1$.
For $k_0\cdots k_{N+m-1}\in \{1,\ldots, \ell\}^{N+m-1}$, set
\begin{align*}
    B_{k_0\cdots k_{N+m-1}}=\prod_{i=0}^{N+m-1}B(x_{k_i},\varepsilon/2)\times \mathcal{X}^{\mathbb{N}_0}.
\end{align*}
For each $0\leq i\leq N-1$ and $\underline{u},\underline{v}\in B_{k_0\cdots k_{N+m-1}}$ we have
\begin{align*}
    d(\sigma^i \underline{u},\sigma^i \underline{v})
    &\leq \sum_{j=0}^{N-i+m-1}\frac{d_{\mathcal{X}}(u_{i+j}, v_{i+j})}{2^{j+1}}+\sum_{j\geq N-i+m-1}\frac{1}{2^{i+1}}\\
    & <\frac{\varepsilon}{2}+\frac{\varepsilon}{2}=\varepsilon.
\end{align*}
Hence we have ${\rm diam}_{d_N} B_{k_0\cdots k_{N+m-1}}<\varepsilon$ and this implies
\begin{align*}
    \#(\mathcal{X}^{\mathbb{N}_0}, d_N, \varepsilon)\leq \ell^{N+m-1}.
\end{align*}
Hence we have
\begin{align*}
    S(\mathcal{X}^{\mathbb{N}_0}, \sigma, d,\varepsilon)
    &=\lim_{N\to\infty} \frac{1}{N}\log\#(\mathcal{X}^{\mathbb{N}_0}, d_N, \varepsilon) \\
    &\leq \lim_{N\to\infty}\frac{N+m-1}{N}\log \ell
    =\log \ell.
\end{align*}

Let $\varepsilon_0=2^{-2}$ and $\ell_0= \#_{{\rm sep}}(\mathcal{X},d_{\mathcal{X}}, \varepsilon_0)$.
Take $\kappa\geq 2$.
For each $n\in \mathbb{N}_0$ set $\varepsilon_n=\varepsilon_0\kappa^{-n}$.
Take $n$ such that $\kappa^{-(n+1)}\leq \varepsilon<\kappa^{-n}$.
Taking into account the relation between the cardinalities of spanning sets and separating sets, we have by Lemma \ref{lem:upper_sep},
\begin{align*}
 %   \#_{{\rm span}}(\mathcal{X},d_{\mathcal{X}}, \varepsilon/2)
   \ell &\leq \#_{{\rm sep}}(\mathcal{X},d_{\mathcal{X}}, \varepsilon/4)\\
    &\leq \#_{{\rm sep}}(\mathcal{X},d_{\mathcal{X}}, \varepsilon_{n})
    \leq (C\kappa^s)^{n} \ell_0.
\end{align*}

Then we have
\begin{align*}
    \frac{S(\mathcal{X}^{\mathbb{N}_0}, \sigma, d, \varepsilon)}{(n+1)\log \kappa}
    \leq \frac{\log (C\kappa^s)^{n}\ell_0}{(n+1)\log \kappa}
    =\frac{n}{n+1}\frac{\log C+s\log \kappa}{\log \kappa}+\frac{\log \ell_0}{(n+1)\log \kappa}.
\end{align*}
Taking  $\limsup_{\varepsilon\to 0}\ (\limsup_{n\to\infty})$, we have 
\begin{align*}
    \overline{{\rm mdim}}_M(\mathcal{X}^{\mathbb{N}_0}, \sigma, d)
    \leq \frac{\log C+s\log \kappa}{\log \kappa}.
\end{align*}
Letting $\kappa\to\infty$, we complete the proof.
\end{proof}

%%%%%%%%%%%%%%%%%%%%%%%%%%%%%%%%%%%%%%%%%%%%%%%%
\vspace*{33pt}

\noindent
\textbf{Acknowledgement.}~ 
The first author was partially supported by JSPS KAKENHI Grant Number JP26K17005. 
The second author was partially supported by JSPS KAKENHI Grant Number JP21K13816.
The authors would like to thank Masaki Tsukamoto and Masayuki Asaoka for helpful comments.

\vspace{11pt }
\noindent
\textbf{Data Availability.}~
Data sharing not applicable to this article as no datasets were generated or analyzed during the current study.

%%%%%%%%%%%%%%%%%%%%%%%%%%%%%%%%%%%%%%%%%%%%%%%%%%%%%%%%

\bibliographystyle{alpha}
\bibliography{DV_Gibbs.bib}

\begin{bibdiv}
\begin{biblist}

\bib{BCLMS2011}{article}{
      author={Baraviera, A.~T.},
      author={Cioletti, L.~M.},
      author={Lopes, A.~O.},
      author={Mohr, Joana},
      author={Souza, Rafael~Rigao},
       title={On the general one-dimensional {{\(XY\)}} model: positive and
  zero temperature, selection and non-selection},
    language={English},
        date={2011},
        ISSN={0129-055X},
     journal={Rev. Math. Phys.},
      volume={23},
      number={10},
       pages={1063\ndash 1113},
}

\bib{BG2000}{article}{
      author={Bylund, Per},
      author={Gudayol, Jaume},
       title={On the existence of doubling measures with certain regularity
  properties},
    language={English},
        date={2000},
        ISSN={0002-9939},
     journal={Proc. Am. Math. Soc.},
      volume={128},
      number={11},
       pages={3317\ndash 3327},
         url={hdl.handle.net/2445/7706},
}

\bib{Bow75}{book}{
      author={Bowen, Rufus},
       title={Equilibrium states and the ergodic theory of {Anosov}
  diffeomorphisms},
    language={English},
      series={Lect. Notes Math.},
   publisher={Springer, Cham},
        date={1975},
      volume={470},
}

\bib{BSS2023}{book}{
      author={B{\'a}r{\'a}ny, Bal{\'a}zs},
      author={Simon, K{\'a}roly},
      author={Solomyak, Boris},
       title={Self-similar and self-affine sets and measures},
   publisher={American Mathematical Society},
        date={2023},
      volume={276},
}

\bib{CLS20}{article}{
      author={Cioletti, Leandro},
      author={Lopes, Artur~O.},
      author={Stadlbauer, Manuel},
       title={Ruelle operator for continuous potentials and {DLR}-{Gibbs}
  measures},
    language={English},
        date={2020},
        ISSN={1078-0947},
     journal={Discrete Contin. Dyn. Syst.},
      volume={40},
      number={8},
       pages={4625\ndash 4652},
}

\bib{CPV24}{article}{
      author={Carvalho, Maria},
      author={Pessil, Gustavo},
      author={Varandas, Paulo},
       title={A convex analysis approach to the metric mean dimension: limits
  of scaled pressures and variational principles},
    language={English},
        date={2024},
        ISSN={0001-8708},
     journal={Adv. Math.},
      volume={436},
       pages={54},
        note={Id/No 109407},
}

\bib{CT06}{book}{
      author={Cover, Thomas~M.},
      author={Thomas, Joy~A.},
       title={Elements of information theory},
    language={English},
     edition={2nd ed.},
   publisher={Hoboken, NJ: John Wiley \& Sons},
        date={2006},
        ISBN={0-471-24195-4; 0-471-74881-1},
}

\bib{D68}{article}{
      author={Dobrushin, R.~L.},
       title={The description of a random field by means of conditional
  probabilities and conditions of its regularity},
    language={Russian},
        date={1968},
        ISSN={0040-361X},
     journal={Teor. Veroyatn. Primen.},
      volume={13},
       pages={201\ndash 229},
}

\bib{Fr14}{article}{
      author={Fraser, Jonathan~M.},
       title={Assouad type dimensions and homogeneity of fractals},
    language={English},
        date={2014},
        ISSN={0002-9947},
     journal={Trans. Am. Math. Soc.},
      volume={366},
      number={12},
       pages={6687\ndash 6733},
         url={hdl.handle.net/10023/5941},
}

\bib{Jen06}{article}{
      author={Jenkinson, Oliver},
       title={Ergodic optimization},
    language={English},
        date={2006},
        ISSN={1078-0947},
     journal={Discrete Contin. Dyn. Syst.},
      volume={15},
      number={1},
       pages={197\ndash 224},
}

\bib{KD1994}{article}{
      author={Kawabata, Tsutomu},
      author={Dembo, Amir},
       title={The rate-distortion dimension of sets and measures},
        date={1994},
     journal={IEEE transactions on information theory},
      volume={40},
      number={5},
       pages={1564\ndash 1572},
}

\bib{LMMS2015}{article}{
      author={Lopes, A.~O.},
      author={Mengue, J.~K.},
      author={Mohr, J.},
      author={Souza, R.~R.},
       title={Entropy and variational principle for one-dimensional lattice
  systems with a general {{\emph{a priori}}} probability: positive and zero
  temperature},
    language={English},
        date={2015},
        ISSN={0143-3857},
     journal={Ergodic Theory Dyn. Syst.},
      volume={35},
      number={6},
       pages={1925\ndash 1961},
}

\bib{LMST09}{article}{
      author={Lopes, A.~O.},
      author={Mohr, J.},
      author={Souza, R.~R.},
      author={Thieullen, Ph.},
       title={Negative entropy, zero temperature and {Markov} chains on the
  interval},
    language={English},
        date={2009},
        ISSN={1678-7544},
     journal={Bull. Braz. Math. Soc. (N.S.)},
      volume={40},
      number={1},
       pages={1\ndash 52},
}

\bib{LinTsu18}{article}{
      author={Lindenstrauss, Elon},
      author={Tsukamoto, Masaki},
       title={From rate distortion theory to metric mean dimension: variational
  principle},
    language={English},
        date={2018},
        ISSN={0018-9448},
     journal={IEEE Trans. Inf. Theory},
      volume={64},
      number={5},
       pages={3590\ndash 3609},
}

\bib{LinTsu19}{article}{
      author={Lindenstrauss, Elon},
      author={Tsukamoto, Masaki},
       title={Double variational principle for mean dimension},
    language={English},
        date={2019},
        ISSN={1016-443X},
     journal={Geom. Funct. Anal.},
      volume={29},
      number={4},
       pages={1048\ndash 1109},
}

\bib{LW20}{article}{
      author={Leplaideur, Renaud},
      author={Watbled, Fr{\'e}d{\'e}rique},
       title={{GENERALIZED CURIE-WEISS POTTS MODELS AND QUADRATIC PRESSURE IN
  ERGODIC THEORY}},
        date={2020-07},
     journal={{Journal of Statistical Physics}},
      volume={181},
      number={1},
         url={https://hal.science/hal-02545170},
}

\bib{Tsu20}{article}{
      author={Tsukamoto, Masaki},
       title={Double variational principle for mean dimension with potential},
    language={English},
        date={2020},
        ISSN={0001-8708},
     journal={Adv. Math.},
      volume={361},
       pages={53},
        note={Id/No 106935},
}

\bib{Tsu25}{article}{
      author={Tsukamoto, Masaki},
       title={On an analogue of the {Hurewicz} theorem for mean dimension},
    language={English},
        date={2025},
        ISSN={0021-2172},
     journal={Isr. J. Math.},
      volume={269},
      number={1},
       pages={241\ndash 268},
}

\bib{VK88}{article}{
      author={Vol'berg, A.~L.},
      author={Konyagin, S.~V.},
       title={On measures with the doubling condition},
    language={English},
        date={1988},
        ISSN={0025-5726},
     journal={Math. USSR, Izv.},
      volume={30},
      number={3},
       pages={629\ndash 638},
}

\end{biblist}
\end{bibdiv}

\end{document}